\documentclass[11pt]{article}

\usepackage{graphicx,amsmath,amsthm,  amssymb, enumerate}
\usepackage{geometry}
\usepackage{latexsym, color}
 \newcommand{\Z}{\mathbb{Z}}

 \newcommand{\R}{\mathbb{R}}

 \newcommand{\E}{\mathbb{E}}
 \newcommand{\EE}{\mathbb{E}}

\newcommand{\cD}{\mathcal{D}}

\newcommand{\cT}{\mathcal{T}}
\newcommand{\cX}{\mathcal{X}}
\newcommand{\cF}{\mathcal{F}}
 
 \newcommand{\AAA}{\mathcal{A}}
  \newcommand{\FFF}{\mathcal{F}}
  
\newcommand{\HHH}{\mathcal{H}}
  \newcommand{\III}{\mathcal{I}}

 \newcommand{\bea}{\begin{eqnarray}}
\newcommand{\ena}{\end{eqnarray}}
\newcommand{\beas}{\begin{eqnarray*}}
\newcommand{\enas}{\end{eqnarray*}} 
 \newtheorem{assumption}{Assumption}

\newtheorem{theorem}{Theorem}
\newtheorem{definition}{Definition}
\newtheorem{proposition}[theorem]{Proposition}
\newtheorem{lemma}[theorem]{Lemma}
\newtheorem{corollary}[theorem]{Corollary}
\newtheorem{example}[theorem]{Example}
\newtheorem{remark}[theorem]{Remark}

\begin{document}

\title{Stein's method for comparison of univariate distributions}

\author{Christophe Ley\footnote{Ghent University,
      Belgium christophe.ley@ugent.be} \and Gesine
    Reinert\footnote{Gesine~Reinert, University of Oxford, United
      Kingdom, reinert@stats.ox.ac.uk} \and Yvik Swan 
\footnote{Universit\'e de Li\`ege, Belgium, yswan@ulg.ac.be}
}

\date{}

\maketitle

\begin{abstract}
We propose a new general version of {Stein's method for univariate distributions}. 
  In particular we propose a canonical definition of the \emph{Stein
    operator} 
  of a probability distribution {which is based on a linear difference
    or differential-type operator}. The resulting \emph{Stein
    identity} highlights the unifying theme behind the literature on
  Stein's method (both for continuous and discrete distributions).
  Viewing the Stein operator as an operator acting on pairs of
  functions, we provide an extensive toolkit for distributional
  comparisons. Several abstract approximation theorems are
  provided. Our approach is illustrated for comparison of several
  pairs of distributions : normal vs normal, sums of independent
  Rademacher vs normal, normal vs Student, {and} maximum of random
  variables vs exponential, Fr\'echet and Gumbel.
\end{abstract}

\tableofcontents
  
 \section{Introduction}
 \label{sec:introduction}
 Stein's method is a popular tool in applied and theoretical
 probability, widely used for Gaussian and Poisson approximation
 problems. The principal aim of the method is to provide quantitative
 assessments in distributional {comparison} statements of the form
 $W \approx Z$ where $Z$ follows a known and well-understood
 probability law (typically normal or Poisson) and $W$ is the object
 of interest. To this end, Charles Stein \cite{S72} in 1972 laid the
 foundation of what is now called ``Stein's method''. For Poisson
 approximation his student Louis Chen \cite{C75} adapted the method
 correspondingly, and hence for Poisson approximation the method is
 often called ``Stein-Chen method'' or ``Chen-Stein method''.  In
 recent years a third very fruitful area of application was born from
 Ivan Nourdin and Giovanni Peccati's pathbreaking idea to intertwine
 Stein's method and Malliavin calculus.  First proposed in
 \cite{NoPe09}, this aspect of the method is now referred to as
 Malliavin-Stein (or Nourdin-Peccati) analysis. For an overview we
 refer to the monographs \cite{Stein1986,BaHoJa92,NP11,ChGoSh11} as
 well as Ivan Nourdin's dedicated webpage
 https://sites.google.com/site/malliavinstein.

 Outside of the Gaussian and Poisson frameworks, {for univariate
   distributions} the method has now also been shown to be effective
 for~: exponential approximation \cite{ChFuRo11,PeRo11}, Gamma
 approximation \cite{Luk1994,Pi04,NoPe09}, binomial approximation
 \cite{ehm1991binomial}, Beta approximation
 \cite{goldstein2013stein,Do14}, the asymptotics of rank distributions
 \cite{fulman2012stein}, inverse and variance Gamma approximation
 \cite{gaunt2014variance,gaunt2013stein}, Laplace approximation
 \cite{PR12}, negative binomial approximation \cite{BGX13} or
 semicircular approximation \cite{GoTi03,GoTi06}.  It can also be
 tailored for specific problems such as preferential attachment graphs
 \cite{PeRoRo12}, the Curie-Weiss model \cite{CS11}, and other models
 from statistical mechanics \cite{EiLo10,eichelsbacher2014rates}. This
 list is by no means exhaustive and we refer the reader to the
 {webpage} https://sites.google.com/site/steinsmethod
 for an accurate overview of this rapidly moving field. {For a target
   distribution for which Stein's method has not yet been developed,
   setting up the method can appear daunting. In this paper we give a
   straightforward yet very flexible framework which not only
   encompasses the known examples but which is also able to cover any
   new distributions which can be given in explicit form. }
 

Broadly speaking, Stein's method consists of  two distinct components, namely
\begin{itemize}
\item[]  \underline{Part~A}: a framework allowing to convert the problem of bounding the 
  error in the approximation of $W$ by $Z$ into a problem of bounding the
  expectation of a certain functional of $W$.  

\medskip

\item[] \underline{Part~B}: a collection of techniques to bound the
  expectation appearing in Part A; the details of these techniques are
  strongly dependent on the properties of $W$ as well as on the form
  of the functional.
 \end{itemize}
For a target probability distribution ${\rm P}$ with support $\III$,
Part A of the method can be sketched as follows. First find a suitable
operator ${\AAA} {:= \mathcal{A}_{\rm P} =} \mathcal{A}_Z$ (called \emph{Stein operator}) and
a wide class of functions $\FFF(\AAA):={\FFF(\AAA_{\rm P}) =} \FFF(\AAA_Z)$ (called
\emph{Stein class}) such that 
\begin{equation}
  \label{steingen} Z \sim {\rm P} \mbox{ if and only if }\E[ \AAA
  f(Z)] = 0 \mbox{ for all }f \in \FFF(\AAA)
\end{equation}
(where $Z \sim \mathrm{P}$ means that $Z$ has distribution
$\mathrm{P}$). 
This equivalence is called a \emph{Stein characterization} of ${\rm
  P}$. Next let $\HHH$ be a measure-determining class on
$\III$. Suppose that for
each $h \in \HHH$ one can find a solution $f = f_h \in \FFF(\AAA)$ of the
\textit{Stein equation}
\begin{equation}
  \label{steineq} h(x) - \E [h(Z)] = \AAA f(x),
\end{equation}
where $Z \sim
{\rm P}$. 
Then, if taking expectations is permitted, we have 
\begin{equation}\label{eq:diffexp}
  \E[h(W)]- \E [h(Z)] = \E \left[ \AAA f(W) \right].
\end{equation}
There exist a number of probability distances (such as
the Kolmogorov, the Wasserstein, and the Total Variation distance)
which can be represented as \emph{integral probability metrics} of the
form 
$$  d_{\mathcal{H}}(W, Z) = \sup_{h \in \mathcal{H}} \left| \E[ h(W)] - \E [h(Z)] \right|,
$$
see \cite[Appendix C]{NP11} or \cite{GS02,Rac91} for
an overview.  From \eqref{eq:diffexp} we get
\begin{equation}
\label{eq:firstepte}
  d_{\mathcal{H}}(W, Z) \le \sup_{f \in \mathcal{F}(\mathcal{H})} \left| \E \left[ \AAA f(W) \right] \right|
\end{equation}
where  $\mathcal{F}(\mathcal{H}) = \left\{ f_h \, | \, h \in
   \mathcal{H} \right\}$  is the
 collection of solutions of \eqref{steineq} for functions $h \in
 \mathcal{H}$. 

 When  only certain features of $W$ are
known, for example that it is a sum of weakly dependent random
variables, then  \eqref{eq:firstepte} is the usual starting point for
Part B of Stein's method. Now suppose that, furthermore, a Stein operator $\AAA_W$ (and a class
$\mathcal{F}(\AAA_W)$) is available for $W$. 
  Suppose also that  $ \FFF(\AAA_Z) \cap
\FFF(\AAA_W) \neq \emptyset$ and choose $\HHH$ such that all
solutions $f$ of the Stein equation \eqref{steineq} for $\AAA_Z$ and
$\AAA_W$ belong to this intersection. Then 
 \beas
  \E [h(W)] - \E[ h(Z) ]&=& \E [\AAA_Z f(W)] \\
  &=& \E[ \AAA_Z f(W)] - \E[ \AAA_W f(W)] \enas 
(because $ \E[ \AAA_W f(W)]  = 0$) and 
\begin{equation}\label{eq:firstepte2}
 d_{\mathcal{H}}(W, Z) \le \sup_{f \in
  \FFF(\AAA_Z) \cap  \FFF(\AAA_W)} | \E [\AAA_W f(W)-\AAA_Z
f(W)]|. 
\end{equation}
Stein \cite{S72} {discovered the magical relation that} 
the
 r.\,h.\,s.~of \eqref{eq:firstepte} or \eqref{eq:firstepte2} provides a handle to assess the
 proximity between the laws of $W$ and $Z$; this is precisely the
 object of Part B of Stein's method. 

 In {many} cases, not only are the functions $f_h$ well-defined, but
 also they possess smoothness properties which render them
 particularly amenable to computations.  Also there exist {many} ways
 by which one can evaluate $ \E \left[ \AAA f(W) \right]$ or
 $ \E [\AAA_W f(W)-\AAA_Z f(W)]$ (even under unfavorable assumptions
 on $W$) including exchangeable pairs (as for example in
 \cite{Stein1986,Ho04,MR2573554,CS11,ChFuRo11,Do14}), biasing
 mechanisms (as in \cite{BaRiSt89,GoRi96,GR97,PeRo11,gaunt2013stein}),
 and other couplings (as in \cite{C75,barbour1985multiple}); see
 \cite{reinert1998couplings,Ro11,chatterjee2014short} for overviews.
 Nourdin and Peccati \cite{NoPe09} paved the way for many elegant
 results in the context of Malliavin calculus, for an overview see
 \cite{NP11}.  See also
 \cite{Ho04,eichelsbacher2008stein,fulman2012stein,goldstein2013stein,Do14,FulmanGoldstein2014b}
 for examples where direct comparison (using the explicit distribution
 of $W$) via \eqref{eq:firstepte2} is used.


 
 Of course the devil is in the detail and the quest for suitable Stein
 operators which are tractable to deal with for the random variables
 in question is essential for the method to be effective. While no
 precise definition of what exactly a \emph{Stein operator} is 
 {most} authors have used
 Stein operators which were differential operators (or difference
 operators in the case of discrete distributions) obtained through a
 suitable variation of one of the {four} following classical
 constructions~:
\begin{itemize}
\item Stein's \emph{density approach} pioneered in \cite{Stein1986}
  relies on the target having an explicit density $p$ (either
  continuous or discrete) and then using integration by parts and
  classical theory of ordinary differential (or difference) equations
  to characterize $p$ (see \cite{CS11,Do14,LS12a,novak,Stein2004} for the
  continuous case, \cite{fulman2012stein,LS13b,novak} for the discrete
  case).
\item Barbour and G\"otze's \emph{generator approach} (see
  \cite{Ba90,G91}) {is} based on classical theory of Markov processes; this
  approach has the added advantage of also providing a probabilistic
  intuition to all the quantities at play. Some references detailing
  this approach for univariate distributions are
  \cite{EdVi12,eichelsbacher2008stein,goldstein2013stein,Ho04,KuTu11,kusuoka2013extension}.
\item Diaconis and Zabell's \emph{orthogonal polynomial approach} (see
  \cite{DZ91}) where they use Rodrigues type formulas, if available,
  for orthogonal polynomials associated with the target distribution.
  See also \cite{S01} as well as \cite{APP11} and related references
  for an extensive study of Stein operators for the Pearson (or Ord)
  family of distributions.
  \item {Probability transformations such as the size bias transformation \cite{BaRiSt89} and the zero bias transformation \cite{GR97} which characterize a distribution through being the unique fixed point of a transformation. See also \cite{GR05} and \cite{PeRoRo12} for examples.} 
\end{itemize}
These three approaches are by no means hermetically separated : often
the operators derived by one method are simple transformations of
those derived by another one.  See for instance \cite{GR05} for a very
general theory on the connection between Stein operators, probability
transforms and orthogonal polynomials.  Other methods of constructing
Stein operators are available.  In \cite{upadhye2014stein} Stein
operators for discrete compound distributions are derived by
exploiting properties of the moment generating function. In
\cite{arras2016stein}, both Fourier and Malliavin-based aproaches are
used to derive operators for targets which can be represented as
linear combinations of independent chi-square random variables.  An
algebraic study of Stein operators {is} initiated in
\cite{gaunt2013stein}, with explicit bounds provided in
\cite{dobler2015iterative}. The
parametric approach presented in \cite{swangeneral,ley2013parametric} laid the foundation to the current work. 

 

 \subsection*{Outline of the paper}

 In this paper we propose a generalization of Stein's density
 approach, in the spirit of \cite{LS13b,ley2013parametric,LS12a} which leads to a
 {canonical} definition of ``the'' differential{-type} operator
 associated to any given density.  {The definition is canonical, or parsimonious,  in the sense that, g}iven a target $p$, 
 {we identify} minimal conditions under which a Stein
 characterization of the form \eqref{steingen} can hold. 
 {Moreover w}e will show with a wealth of examples
 that all the ``useful'' operators mentioned in the introduction can
 be derived as (sometimes not so straightforward) transformations of
 our 
 {operator.}

 In Section \ref{sec:pars-stein-theory} we introduce our 
 {approach} in the
 simplest setting : distributions with continuous probability density
 function.  Two easy applications are provided. In
 Section~\ref{sec:gener-vers-steins} we establish the 
 set-up and introduce our toolbox in all generality.  In Section
 \ref{sec:examples} we discuss different important particular cases
 (which we call standardizations), hereby linking our
 {approach}
 with the
 classical literature on the topic. In
 Section~\ref{sec:grdistr-comp} we provide abstract approximation
 theorems for comparing probability distributions.  In
 Section~\ref{sec:examples-again} we illustrate the power of our
 approach by tackling applications to specific approximation problems.

 \section{The  Stein operator for 
   differentiable {probability density functions}}
\label{sec:pars-stein-theory}
In this section we sketch our approach in the simplest setting : $X$
has absolutely continuous probability density function (pdf) $p$ with
respect to the Lebesgue measure on $\R$. Furthermore we suppose that
$p$ has  interval support $\mathcal{I}$ (i.e.\ $p(x) >0$ for all
$x \in \mathcal{I}$, some real interval which could be unbounded); we denote $a, b$ the {boundary} 
points of $\mathcal{I}$. 

\subsection{The Stein operator}
\label{sec:part-a1-}

\begin{definition}\label{defn:pars-stein-op}
  The Stein class for $p$ is the collection $\mathcal{F}(p)$ of
  functions $f:\R\to\R$ such that (i) $x \mapsto f(x) p(x)$ is
  differentiable, (ii) $x \mapsto (f(x) p(x))'$ is integrable and
  (iii)
  $ \lim_{x \uparrow b}f(x) p(x) = \lim_{x \downarrow a}f(x) p(x) =
  0$.
  The (differential) Stein operator for $p$ is the differential
  operator $\mathcal{T}_{p}$ defined by
  \begin{equation}
    \label{eq:17}
    f \mapsto \mathcal{T}_{p}f :=\frac{(fp)'}{p}
  \end{equation}
with the
  convention that $\mathcal{T}_pf(x)=0$ for $x$ outside of
  $\mathcal{I}$.
\end{definition}

\begin{remark}
  Condition (ii) in Definition \ref{defn:pars-stein-op} may easily be
  relaxed, e.g.\ by only imposing that
  $\int_a^b (f(x) p(x))' dx =: \left[ f(x) p(x) \right]_a^b=0$. This
  condition could also be dispensed with entirely, although this
  necessitates to re-define the operator as 
  $\mathcal{T}_{p}f = {(fp)'}/p - \left[ f(x) p(x) \right]_a^b$. See
  also Remark \ref{rem:definop}.
\end{remark}
\begin{remark}
   It should be stressed that the
assumptions on $f \in \mathcal{F}(p)$ can be quite stringent,
depending on the properties of $p$. There is, for instance, no
guarantee a priori that constant functions $f \equiv 1$ belong to
$\mathcal{F}(p)$, as this requires that $p$ cancels at the edges of
its support and is differentiable with integrable derivative; such
assumptions are satisfied neither in the case of an exponential target
nor in the case of a beta target.
\end{remark}


Obviously we can always expound the derivative in \eqref{eq:17} (at
least formally, because care must be taken with the implicit indicator
functions) to obtain the equivalent expression
 \begin{equation}\label{eq:36}
   \mathcal{T}_pf(x) = f'(x) + \frac{p'(x)}{p(x)}f(x)
 \end{equation} 
 whose form is reminiscent of the operator advocated by
 \cite{Stein2004,CS11}. In our experience, however, operator
 \eqref{eq:36} is unlikely to be useful in that form because most of
 the conditions inherited from $p$ are still implicit in the
 properties of $f$, as illustrated in the following example. 
\begin{example}\label{ex:part-a1--1}
  If $p(x) \propto (x(1-x))^{-1/2} \mathbb{I}[0,1]$ then
  $\mathcal{F}(p)$ is the collection of functions $f : \R \to \R$ such
  that $f(x)/\sqrt{x(1-x)}$ is differentiable with integrable
  derivative and  {with the limiting behavior} $\lim_{x \to 0, 1}f(x)/\sqrt{x(1-x)} = 0$. Operator
  \eqref{eq:36} becomes
  $ \mathcal{T}_pf(x) = f'(x) + (2x-1)/(2x(1-x))f(x)$.  The operator
  is cumbersome but nevertheless well defined at all points
  $x \in [0,1]$ thanks to the conditions on $f \in \mathcal{F}(p)$.
  In particular these conditions ensure that $f(x)$ cancels at 0 and 1
  faster than 
  {$p(x)$} diverges.  Taking functions of the form
  $f(x) = (x(1-x))^{\alpha}f_0(x)$ with $\alpha>1/2$ suffices. For
  instance the choice $\alpha=1$ yields the operator
  $ \mathcal{A}_pf_0(x) = \mathcal{T}_pf(x) = x(1-x)f_0'(x)+\left(
    \frac{1}{2} -x\right) f_0(x)$
  used in \cite{goldstein2013stein,Do14} for Beta approximation.
\end{example}
The pair $(\mathcal{T}_p, \mathcal{F}(p))$ is uniquely associated to
$p$. By choosing to focus on different subclasses
$\mathcal{F}(\mathcal{A}_p) \subset \mathcal{F}(p)$ one obtains
different operators acting on different sets of functions.
We call the passage from
$(\mathcal{T}_p, \mathcal{F}(p))$ to
$(\mathcal{A}_p, \mathcal{F}(A_p))$ a \emph{parameterization} of the
Stein operator. {{There}} remains full freedom in the choice of
this explicit form and it remains necessary to further understand the
properties of $p$ in order to select those functions
$f \in \mathcal{F}(p)$ for which \eqref{eq:36} will assume the most
tractable expression.  In Example~\ref{ex:part-a1--1} this is achieved
by a simple transformation of the test functions; in other cases the
transformations are much more complex and the resulting operators are
not even necessarily of first order.

\begin{example}[Kummer-$U$ distribution] \label{ex:kummerv1} Let $U(a,
  b, z)$ be the unique solution of the differential equation $z
  d^2U/dz^2 + (b-z) dU/dz -aU = 0$. Then $U(a, b, z)$ is the
 confluent hypergeometric function of the second kind (also known as
 the Kummer $U$ function).  A random variable  $X$ follows the
 Kummer-$U$ distribution $K_s$ if its  density is
 \begin{equation*}
      \kappa_s(x) = \Gamma(s) \sqrt{\frac{2}{s \pi}} {\exp}\left(
     \frac{-x^2}{2s} \right) V_s(x) \mathbb{I}(x\in(0,\infty)), \quad s\geq 1/2,
 \end{equation*}
 with $\Gamma(s)$ the Gamma function and
 $V_s(x) = U \left( s-1, \frac{1}{2}, \frac{x^2}{2s} \right)$. The
 class $\mathcal{F}(\kappa_s)$ contains all differentiable functions
 such that
 $\lim_{x\rightarrow0\,{\rm or}\,\infty}f(x)\kappa_s(x) = 0$. As noted
 in \cite{PeRoRo12}, the canonical Stein operator (as given in
 \eqref{eq:36}) is cumbersome.  
One can show by direct  computations that for  differentiable $f_0$
we have 
\begin{equation*}
 \frac{\left( \kappa_s(x)
      \frac{(f_0(x)V_s(x))'}{V_s(x)}  \right)'}{\kappa_s(x) } = s f_0''(x)
  -x f_0'(x) - 2(s-1) f_0(x)  =: \mathcal{A}_0(f_0)(x)
\end{equation*}
for $x>0$, which suggests to consider functions
$f \in \mathcal{F}(\kappa_s)$ of the form
$$f(x) = \frac{(f_0(x)V_s(x))'}{V_s(x)},$$ hereby providing a new
derivation of the \emph{second order} operator given in \cite[Lemma
3.1, Lemma 3.2]{PeRoRo12} where Stein's method was first set up for
this distribution.
\end{example}


\subsection{The generalized Stein covariance identity}
\label{sec:part-a2-}

Given a function $f \in \mathcal{F}(p)$, we now introduce a
\emph{second} class of functions which contains all $g : \R \to \R$
which satisfy the integration by parts identity :
 \begin{equation}
   \label{eq:7}
\int_a^b g(x) (f(x) p(x))' dx =   - \int_a^b g'(x)(f(x) p(x)) dx. 
 \end{equation}
 It is easy to deduce conditions
 under which \eqref{eq:7} holds; these  are summarized in
 the next definition. 
 \begin{definition}
   Let $p$ be as above. To each $f \in \mathcal{F}(p)$ we associate
   $\mathcal{G}(p, f)$, the collection of functions such that\\
    (i) $x
   \mapsto |g(x) (f(x) p(x))'|$,  $x \mapsto |g'(x)(f(x) p(x))|$ are
   both  integrable on $\mathcal{I}$; \\
   (ii) $\left[ g(x)\, f(x) \, p(x)
   \right]_a^b=0$. \\
    We also define 
   $\mathcal{G}(p) = \bigcap_{f \in \mathcal{F}(p)} \mathcal{G}(p,
   f),$
and call these functions the test functions for $p$. 
 \end{definition}
 If $\mathcal{F}(p)$ is not empty then neither are $\mathcal{G}(p,f)$
 and $\mathcal{G}(p)$ because they must contain the constant function
 $g \equiv 1$.  Rewriting identity \eqref{eq:7} in terms of the Stein
 pair $(\mathcal{T}_p, \mathcal{F}(p))$ leads to the \emph{generalized
   Stein covariance identity}
\begin{equation}\label{eq:33}
 \E \left[ g(X) \mathcal{T}_{p}f(X) \right] = - \E \left[
   g'(X) f(X) \right] \mbox{ for all } f \in \mathcal{F}(p) \mbox{ and
 } g \in \mathcal{G}(p, f).
\end{equation}
This identity generalizes several fundamental probabilistic
integration by parts formulas.  For instance if, on the one hand,
$f \equiv 1 \in \mathcal{F}(p)$ then $\mathcal{G}(p, 1)$  contains
all $g : \R \to \R$ that are absolutely continuous with compact
support and
\begin{equation*}
\E \left[ g(X) \rho(X) \right] = - \E \left[
   g'(X)  \right] \mbox{ for all }  g \in \mathcal{G}(p, 1),
\end{equation*} 
with $\rho = \mathcal{T}_p1$ the score function of $X$. If, on the
other hand, $\E[X] = \mu$ is finite then choosing $h(y) = E[X]-y$ leads to Stein's
classical covariance identity
\begin{equation*}
 \E \left[ (X-\mu) g(X) \right] = \E \left[\tau_p(X) g'(X)  \right]
 \mbox{ for all } g \in \mathcal{G}(p,  \tau_p)
\end{equation*}
with 
\begin{equation}\label{eq:steinkernn}
  \tau_p(x) = \frac{1}{p(x)}\int_x^{\infty}(y-\nu)
  p(y)dy
\end{equation}
the so-called \emph{Stein kernel} of $p$ and $\mathcal{G}$ the
corresponding collection of functions; it is easy to see that it
suffices that $g$ be differentiable and bounded at the edges of the
support of $\mathcal{I}$.  This approach was first studied in
\cite{Stein1986} (see also \cite{Do14,KuTu11,APP11}).
\begin{remark}
  Equation \eqref{eq:33} leads us to an alternative definition of the
  Stein operator \eqref{eq:17} as some form of skew-adjoint operator
  to the derivative with respect to integration in $pdx$.
\end{remark}

\subsection{Stein characterizations} 
\label{sec:part-a3-}

In Section~\ref{sec:stein-char} we will show that, under reasonable assumptions on $p$, the classes
$\mathcal{F}(p)$ and $\mathcal{G}(p)$ are sufficiently large to ensure
that \eqref{eq:33} also characterizes the distribution $p$. This
realization leads to a 
{collection} of versions of the 
Stein characterization \eqref{steingen}. For example,    we shall prove  that 
\begin{eqnarray}\label{eq:stec1}
\lefteqn{\mbox{for each }  g \in \mathcal{G}(p), } \nonumber \\
Y &\sim &p \Longleftrightarrow   \E \left[  g(Y) \mathcal{T}_{p}f(Y)\right] = - \E \left[
   g'(Y)f(Y)  \right] \mbox{ for all }  f \in \mathcal{F}(p);
\end{eqnarray}
and 
\begin{eqnarray}\label{eq:stec2}
  \lefteqn{\mbox{for each } f \in \mathcal{F}(p),} \nonumber \\
    Y &\sim& p \Longleftrightarrow \E \left[  g(Y) \mathcal{T}_{p}f(Y)\right] = - \E \left[
    g'(Y)f(Y)  \right] \mbox{ for all }  g \in \mathcal{G}(p,f).
\end{eqnarray}
We refer to Section~\ref{sec:stein-char} for more information as well
as a precise statement of the conditions on $p$ under which such
characterizations hold.

The freedom of choice for test functions $f$ and $g$ implies that many
different characterizations can be immediately deduced from
\eqref{eq:stec1}, \eqref{eq:stec2} or more generally from
\eqref{eq:33}. For example taking $g=1$ in \eqref{eq:stec1} we obtain
\begin{equation}
  \label{eq:stech1}
Y \sim p \Longleftrightarrow  \E \left[f'(X) + f(X) \frac{p'(X)}{p(X)} \right] =
0 \mbox{ for all } f \in \mathcal{F}(p)
\end{equation}
with $\mathcal{F}(p)$ the functions such that $(fp)'$ is integrable
with integral 0. 
If one is  allowed to take 
$f=1$ in \eqref{eq:stec2}
then   we deduce the characterization
\begin{equation}
  \label{eq:stech2}
  Y \sim p \Longleftrightarrow   \E \left[ g(Y) \frac{p'(Y)}{p(Y)} \right] = - \E \left[
   g'(Y) \right] \mbox{ for all }  g \in \mathcal{G}(p,1), 
\end{equation}
with $ \mathcal{G}(p,1)$ the functions such that $g p'$ and $g' p$ are
integrable and $gp$ has integral 0.  Although the difference between
\eqref{eq:stech1} and \eqref{eq:stech2} may be subtle, the last
characterization is more in line with the classical literature on the
topic to be found e.g.\ in \cite{CS11}'s general approach (the
specific conditions outlined in \cite{CS11} for their {approach} 
to work out guarantee that $1 \in \mathcal{F}(p)$).


\subsection{Stein equations and Stein factors}
\label{sec:stein-equat-stein}

The heuristic behind Stein's method outlined in the Introduction is
that if $X\sim p$ is characterized by
$\E \left[ \mathcal{A}_Xf(X) \right] = 0$ over the class
$\mathcal{F}(\mathcal{A}_X)$ then
$\Delta_f(Y, X) := \left| E \left[ \mathcal{A}_X f(Y) \right] \right|$
ought to be a good measure of how far the law of $Y$ is from that of
$X$. Considering equations such as \eqref{eq:diffexp} leads to the
conclusion that indeed $\sup_f\Delta(Y, X)$ provides a bound on all
integral probability metrics such as \eqref{eq:firstepte}.

A similar reasoning starting from the generalized Stein covariance
identity \eqref{eq:33} encourages us to consider generalized Stein
equations of the form
\begin{equation}
  \label{eq:1}
   g(x) \mathcal{T}_{p}f(x) +    g'(x) f(x) = h(x) - \E[h(X)]
\end{equation}
(these are now equations in two unknown functions) and the
corresponding quantities
\begin{equation}
  \label{eq:5}
  \Delta_{f, g}(X, Y) = \left|  \E \left[ g(Y) \mathcal{T}_{p}f(Y) +    g'(Y) f(Y)  \right]  \right|
\end{equation}
for $f \in \mathcal{F}(p)$ and $g \in \mathcal{G}(p, f)$. 

There are many ways to exploit the freedom of choice of test functions
$(f, g)$ in \eqref{eq:5}. A clear aim is to choose these functions in
such a way that the expression is as manageable as possible and to
this end it is natural to consider $f \in \mathcal{F}(p)$ such that
\begin{equation}\label{eq:setcinv}
  \mathcal{T}_p(f) = h
\end{equation}
for some well-chosen $h$. Obviously for \eqref{eq:setcinv} to make
sense it is necessary that $h$ have mean 0 and, in this case, it is
easy to solve this first order equation, at least
formally. Introducing the class $\mathcal{F}^{(0)}(p)$ of functions
with $p$-mean 0 we are now in a position to introduce the
\emph{inverse Stein operator}
\begin{equation}
  \label{eq:12}
  \mathcal{T}_p^{-1} : \mathcal{F}^{(0)}(p) \mapsto \mathcal{F}(p)  :
  h \mapsto \frac{1}{p(x)} \int_a^x h(u) p(u) du.
\end{equation}
Similarly as with the differential Stein operator $\mathcal{T}_p$, the
integral operator $\mathcal{T}_p^{-1}$ is uniquely associated to
$p$. 
\begin{example}
  The Stein kernel \eqref{eq:steinkernn} is $\mathcal{T}_p^{-1}h$
with $h$ the (recentered) identity function.
\end{example}

In general one will choose $f$ and $g$ in such a way as to ensure that
(i) both $\mathcal{T}_pf$ and $f$ have agreeable expressions, and (ii)
solutions to \eqref{eq:1} have good properties, hereby ensuring that
\eqref{eq:5} is amenable to computations. We will show in Sections~\ref{sec:grdistr-comp} 
 and \ref{sec:examples-again} that this is
the case for a wide variety of target distributions. Given 
$\mathcal{H} \subset \mathcal{F}^{(0)}$,  constants such as 
 \begin{equation}
  \label{eq:13}
\sup_{h \in \mathcal{H}} \left\|  \mathcal{T}_p^{-1}h
\right\|_{\infty}, \, \sup_{h \in \mathcal{H}} \left\|  \left(
    \mathcal{T}_p^{-1}h \right)'  \right\|_{\infty}
\end{equation}
will play an important role in the success of the method. These are
usually referred to as the \emph{Stein factors} of $p$, and there is
already a large body of literature dedicated to their study under
various assumptions on $p$, see e.g.\
\cite{brown1995stein,RO12,BGX13,dobler2015iterative}.

\subsection{Comparing {probability densities} by comparing Stein operators}

Now let $X_1$ and $X_2$ have densities $p_1, p_2$ with supports
$\mathcal{I}_1, \mathcal{I}_2$ and Stein pair
$(\mathcal{T}_1, \mathcal{F}_1)$ and $(\mathcal{T}_2, \mathcal{F}_2)$,
respectively.  Equation \eqref{eq:1} leads to an ensemble of
\emph{Stein equations} for $X_i, i=1, 2$ of the form
\begin{equation}
  \label{eq:26}
    h(x) - \E[ h(X_i)] = g'(x) f(x) + g(x) \mathcal{T}_if(x)
\end{equation}
whose solutions are now pairs
$(f, g) \in \mathcal{F}(p_i) \times \mathcal{G}(p_i)$. 
Given a sufficiently regular function $h$ then any pair
$f_i, g_i\in \mathcal{F}(p_i) \times \mathcal{G}(p_i)$ satisfying
\begin{equation}\label{eq:20}
 f_i(x)  g_i(x) = \frac{1}{p_i(x)} \int_{a_i}^x p_i(u) \left(  h(u)
    - \E[ h(X_i)] \right) du
\end{equation}
(with $a_i, i=1, 2$ the lower edge of $\mathcal{I}_i$) is a solution
to \eqref{eq:26} for $ i=1, 2$. Functions such as the one on the rhs
of \eqref{eq:20} have been extensively studied, see e.g.\
\cite{Stein1986,KuTu11}.

There are many starting points from here. For example taking differences
between {E}quations \eqref{eq:26} for $i=1, 2$ leads to the unusual
identity 
\begin{align}\label{eq:35} 
\lefteqn{\E[h(X_2)] - \E[h(X_1)] } \nonumber \\
&  = \left( g_1'(x)f_1(x)-g_2'(x)f_2(x) \right) + 
   \left( g_1(x)\mathcal{T}_1f_1(x)-  g_2(x)\mathcal{T}_2f_2(x)
                          \right) 
\end{align}
 for all
$x \in \mathcal{I}_1\cap \mathcal{I}_2$ and  all $(f_i, g_i) \in \mathcal{F}(p_i) \times \mathcal{G}(p_i)$
which satisfy \eqref{eq:20}.  
Another approach is to pick $(f_1, g_1)$ solution to \eqref{eq:20} and 
$(f_2, g_2) \in \mathcal{F}(p_2) \times \mathcal{G}(p_2)$ (which
ensures that
$ \E \left[ g_2'(X_2) f_2(X_2) + g_2(X_2) \mathcal{T}_2f_2(X_2)
\right] = 0$)
and to take expectations in $X_2$ on both sides of \eqref{eq:26},
yielding
\begin{eqnarray}
  \label{eq:24}
 \lefteqn{ \E[h(X_2)] - \E[h(X_1)] } \nonumber \\
 &  = &  \E \left[   g_1'(X_2) f_1(X_2) + g_1(X_2)
                            \mathcal{T}_1f_1(X_2) \right] \nonumber  \\
 &  =  & \E \left[   g_1'(X_2) f_1(X_2)  -   g_2'(X_2) f_2(X_2)  \right] \nonumber  \\
 && -    \E \left[    g_1(X_2)
                            \mathcal{T}_1f_1(X_2)  -  g_2(X_2)
                            \mathcal{T}_2f_2(X_2) \right],  
\end{eqnarray}
under the additional  assumption that all expectations exist. 
 Identity \eqref{eq:24} is a powerful
starting point for stochastic approximation problems, as one can
handpick the functions $f_i, i=1, 2$ and $g_i, i=1, 2$ best suited to
the problem under study.  
\begin{itemize}
\item Assume that  $f_1 = f_2 = 1$ is permitted and that  $g_1$,
  defined in \eqref{eq:20}, belongs to $\mathcal{G}(p_2)$.  Then from \eqref{eq:24} we deduce that 
 \begin{align*} \E[h(X_2)] -
  \E[h(X_1)] & = \E \left[ g_1(X_2) \left( 
                            \rho_2(X_2)  - 
                            \rho_1(X_2)  \right)\right]
\end{align*}
where $\rho_i$ is the score function of $X_i$. This identity (which
holds as soon as $g_1 \in \mathcal{F}(p_2)$) in turn leads to the
Fisher information inequalities studied, e.g., in
\cite{Sh75,MR2128239,LS12a}.
\item Assume that  $X_1, X_2$ both have mean $\nu$ and pick $f_1, f_2$
  such that  $\mathcal{T}_1f_1 =  \mathcal{T}_2f_2 = x -\nu$. Let
  $g_1$ be the corresponding function from \eqref{eq:20} and assume
  that  $g_1\in \mathcal{G}(p_2)$. Then 
 \begin{equation} \label{eq:29}
\E[h(X_2)] - \E[h(X_1)]   =   
\E \left[  g_1'(X_2) \left( \tau_1(X_2)-  \tau_2(X_2) \right) \right]
\end{equation} 
where $\tau_i$ is the Stein kernel of $X_i$. From here one readily
recovers the key inequalities from \cite{CPU94,CPP01}. This is also the starting
point of the Nourdin-Peccati approach to Stein's method \cite{NP11}.
\end{itemize}
Many other identities can be obtained
. 
We have recently applied this result to the computation of explicit
bounds in a problem of Bayesian analysis, see \cite{ley2015distances}.
Several applications will be provided in Sections
\ref{sec:grdistr-comp} and \ref{sec:examples-again}. We conclude
this section with two  easy applications.

\subsection{Application 1 : rates of convergence to the Fr\'echet distribution}
\label{sec:rates-conv-frech}
Let $X_{\alpha}$ follow the Fr\'echet distribution with tail index
$\alpha$ so that
$P(X_{\alpha}\le x) =: \Phi_{\alpha}(x) =
\mathrm{exp}(-x^{-\alpha})\mathbb{I}(x\ge0)$.
Applying the theory outlined in the previous sections, the Stein class
$\mathcal{F}(\alpha)$ for the Fr\'echet is the collection of all
differentiable functions {$f$} on $\R$ such that
$\lim_{x \to +\infty}f(x)x^{-\alpha-1} e^{-x^{-\alpha}}=\lim_{x \to
  0}f(x)x^{-\alpha-1} e^{-x^{-\alpha}} =0.$
We restrict our attention to functions of the form
$ f(x) = x^{\alpha+1} f_0(x) $.  In this parameterization the
differential Stein operator becomes
\begin{equation}
  \label{eq:28}
  \mathcal{A}_{\alpha}f_0(x) = x^{\alpha+1}f_0'(x) + \alpha
f_0(x).
\end{equation}
The generalized Stein equation \eqref{eq:26} with $g = 1$ reads 
$ x^{\alpha+1}f_0'(x) + \alpha f_0(x) = h(x) - \E h(X_{\alpha})$ and, given
$h(x) = \mathbb{I}(x \le z)$, has unique bounded solution 
\begin{equation}
  \label{eq:fresteq}
  f_z(x) = \frac{1}{\alpha} \left( {\Phi_{\alpha}(x \wedge z)
  }/{\Phi_{\alpha}(x)} -\Phi_{\alpha}(z) \right). 
\end{equation}
This function is continuous and differentiable everywhere except at
$x=z$; it satisfies $0 \le \alpha f_{z}(x) \le 1$ for all $x, z\ge0$
as well as $\lim_{x\to +\infty}f_z(x) = 0$.

Next take $F(x) = (1-x^{-\alpha})\mathbb{I}(x \ge 1)$ the Pareto
distribution and  for $n\ge1$ consider the
random variable $W_n = M_n/n^{1/\alpha}$. Its  probability density
function is 
$p_n(x) = \alpha x^{-\alpha-1} \left( 1-{x^{-\alpha}}/{n}
\right)^{n-1}$
on $[n^{-1/\alpha}, +\infty)$. For each $n$ the random variable $W_n$
has a Stein pair $(\mathcal{T}_n, \mathcal{F}(n))$, say. In order to
compare with the Fr\'echet distribution we consider the
standardization
\begin{equation*}
  \mathcal{A}_n(f_0)(x) = \frac{(x^{\alpha+1} f_0(x) p_n(x))'
  }{p_n(x)} = x^{\alpha+1} f_0'(x) + \alpha  \frac{n-1}{n} \left( 1-\frac{x^{-\alpha}}{n} \right)^{-1} f_0(x)
\end{equation*}
with $f_0$ an absolutely continuous function such that
$$\lim_{x\to +\infty}x^{\alpha+1} f_0(x) p_n(x) = \lim_{x \to
  n^{-1/\alpha}} x^{\alpha+1} f_0(x) p_n(x) =0.$$
The function $f_z$ given in \eqref{eq:fresteq} satisfies these two
constraints. Hence $\E \left[ \mathcal{A}_n(f_z)(W_n) \right]=0$ and
from \eqref{eq:24} we get in this particular case
\begin{equation*}
P(W_n \le z) - \Phi_{\alpha}(z)
   = \alpha E \left[ f_z(W_n) \left( 1-\frac{n-1}{n} \left(
         1-\frac{W_n^{-\alpha}}{n} \right)^{-1}  \right) \right]. 
\end{equation*}
The function $x \mapsto 1-\frac{n-1}{n} \left( 1-\frac{x^{-\alpha}}{n}
      \right)^{-1}$ is negative for all $x \ge n^{-1/\alpha}$. Also, 
      it is easy to compute  explicitly 
$  E \left[ \frac{n-1}{n} \left( 1-\frac{W_n^{-\alpha}}{n}
      \right)^{-1} -1   \right] = \frac{2}{n-1}\left(
    1-\frac{1}{n} \right)^n 
$. We deduce the upper bound 
\begin{align*}
  \sup_{z \in \R} |P(W_n \le z) - \Phi_{\alpha}(z)| \le
  \frac{2e^{-1}}{n-1}. 
\end{align*} 
More general bounds of the same form can be readily obtained  for
maxima of independent random variables satisfying adhoc tail
assumptions. 

\subsection{Application 2 : a CLT for random variables with a Stein
  kernel}
\label{sec:application-2-}

Let $X_1, \ldots, X_n$ be independent  centered {continuous}  random variables with
unit variance and 
Stein kernels $\tau_1, \ldots, \tau_n$ as given by
\eqref{eq:steinkernn}.  Also let $Z$ be a standard normal random
variable independent of all else. The standard normal random variable
(is characterized by the fact that it) has constant Stein kernel
$\tau_Z(x) = 1$. Finally let $W = \frac{1}{\sqrt{n}} \sum_{i=1}^n X_i$. We
will prove in Section \ref{sec:sum} that
\begin{equation}
  \label{eq:25}
  \tau_W(w) = \frac{1}{n} \sum_{i=1}^n \E \left[ \tau_i(X_i) \, | \,
    W=w \right]
\end{equation}
(see Proposition \ref{prop:sumwithscale}) which we can use in \eqref{eq:29} (setting $X_2 = W$ and $X_1 = Z$) to
deduce that
\begin{align*}
  \E \left[ h(W) \right] - \E \left[ h(Z) \right] & = \frac{1}{n}
 \E \left[ g_1'(W)    \sum_{i=1}^n
\left( 1 - \tau_i(X_i)  \right) \right]\\
&  \le\frac{1}{n}  \sqrt{ \E \left[ g_1'(W)^2 \right] \E \left[\left(  \sum_{i=1}^n
\left( 1 - \tau_i(X_i)  \right) \right)^2 \right]}.
\end{align*} 
Classical results on Gaussian Stein's method  
{give} that $\|g_1'\|_{\infty} \le 1$ if $h(x) =
\mathbb{I}(x \le z)$, see   \cite[Lemma
2.3]{ChGoSh11}.   Also, using the fact that  $\E [ 1 -
\tau_i(X_i)]= 0$ for all $i=1, \ldots, n$ as well as   
$\E \left[ \left( \tau_i(X_i) -1 \right)^2 \right] = \E \left[
  \tau_i(X_i)^2 \right] - 1,
$
we get 
\begin{align*}
  \E \left[\left(  \sum_{i=1}^n
\left( 1 - \tau_i(X_i)  \right) \right)^2 \right] & = \mbox{Var}  \left( \sum_{i=1}^n
\left( 1 - \tau_i(X_i)  \right)  \right) 
                                                    = 
\sum_{i=1}^n \left( \E \left[ \tau_i(X_i)^2 \right]-1 \right).
\end{align*}
If the $X_i$ are i.i.d.\, {then}  we finally conclude {that} 
\begin{align}\label{eq:CLTSTk}
\sup_z \left| P(W \le z) - P(Z \le z)  \right| \le \frac{1}{\sqrt{n}} {\sqrt{\left( \E \left[ \tau_1(X_1)^2 \right]-1 \right)}}. 
\end{align}
Of course \eqref{eq:CLTSTk} is for illustrative purposes only because the requirement that the $X_i, i=1, \ldots, n$ possess a
Stein kernel is very restrictive (even more restrictive than the
existence of a fourth moment). In this application it is  assumed
  that $W$ has a continuous distribution; this assumption is not
  necessary because Stein kernels can  be defined for {any univariate distribution.} 
  We will provide a general version of
\eqref{eq:CLTSTk}  in Section \ref{sec:sum}.

\section{The canonical Stein operator}
\label{sec:gener-vers-steins}

In this section we lay down the foundations and set the framework for
our general theory of {canonical} Stein operators.
\subsection{The setup}
\label{sec:setup}

Let $(\cX, {\mathcal{B}}, \mu)$ be a measure space with $\mathcal{X} \subset
\R$ 
(see Remark \ref{rem:dim}).
Let $\mathcal{X}^{\star}$ be the set of real-valued functions on
$\cX$.  We require the existence of a linear operator
\begin{equation*}
  \mathcal{D} :
dom(\cD)\subset \mathcal{X}^{\star} \to im(\mathcal{D})
\end{equation*}
 such that
$dom(\cD)\setminus \left\{ 0 \right\}$ is not empty. As is standard we
define  
\begin{equation*}
  \mathcal{D}^{-1} : im(\mathcal{D}) \to dom(\mathcal{D})
\end{equation*}
as the linear operator which sends any $h = \mathcal{D}f$ onto
$f$. This operator is a right-inverse for $\mathcal{D}$ in the sense
that 
$  \mathcal{D} \left( \mathcal{D}^{-1}h \right)=h$
for all $h \in im (\mathcal{D})$ whereas, for $f\in dom(\mathcal{D})$,
$\mathcal{D}^{-1} \left( 
  \mathcal{D}f \right)$ is only defined up to addition with an element
of $ker(\mathcal{D})$. We impose the following assumption. 
\begin{assumption} \label{ass:prodrule}
There exists a linear operator $\mathcal{D}^{\star}
  : dom(\cD^{\star})\subset \mathcal{X}^{\star} \to
  im(\mathcal{D}^{\star})$ and a constant $l:=l_{\mathcal{X},
    \mathcal{D}}$ such that 
  \begin{equation}
    \label{eq:2}
    \mathcal{D}(f(x)g(x+l))  = g(x) \mathcal{D}f(x)
 + f(x) \mathcal{D}^{\star}g(x)
  \end{equation}
for all $(f, g) \in dom(\mathcal{D})\times dom (\mathcal{D}^{\star})$
and for all $x \in \mathcal{X}$.  
\end{assumption}

Assumption \ref{ass:prodrule} guarantees that operators $\mathcal{D}$
and $\mathcal{D}^{\star}$ are skew-adjoint in 
the sense that
\begin{equation}
  \label{eq:4}
  \int_{\mathcal{X}} g \mathcal{D}f  d\mu = - \int_{\mathcal{X}} f
  \mathcal{D}^{\star}g d\mu
\end{equation}
for all $(f, g) \in dom(\mathcal{D})\times dom
(\mathcal{D}^{\star})$ such that
$g \mathcal{D}f\in L^1(\mu)$, or $f\mathcal{D}^{\star}g \in
L^1(\mu)$, and   $ \int_{\mathcal{X}}\mathcal{D}(f(\cdot)g(\cdot+l))
d\mu  = 0$.

\begin{example}[Lebesgue measure]\label{ex:cpont}
  Let $\mu$ be the Lebesgue measure on $\mathcal{X}=\R$  and take $\cD
  $ the usual strong derivative. Then 
  \begin{equation*}
    \mathcal{D}^{-1}f(x) = \int_\bullet^x f(u) du
  \end{equation*}
is the usual antiderivative. Assumption \ref{ass:prodrule} is satisfied  with $
  \cD^{\star}=\cD$ and $l=0$. 
\end{example}

\begin{example}[Counting measure]\label{ex:discm}
  Let $\mu$ be the counting measure on $\mathcal{X}=\Z$
and take $\cD=\Delta^+$, the forward difference
  operator $\Delta^+f(x) = f(x+1) - f(x)$.   Then 
  \begin{equation*}
    \mathcal{D}^{-1}f(x) = \sum_{k=\bullet}^{x-1}f(k).
  \end{equation*}
Also we have the
discrete product rule 
\begin{equation*}
  \Delta^{+}(f(x)g(x-1))     =  g(x) \Delta^{+}f(x)+ f(x)
  \Delta^{-}g(x)
\end{equation*}  
for all $f, g \in \Z^{\star}$ and all $x \in \Z$.  Hence Assumption
\ref{ass:prodrule} is satisfied with $\mathcal{D}^{\star} =
\Delta^{-}$, the backward difference operator,   and $l=-1$.
\end{example} 

\begin{example}[Counting measure on the grid]\label{ex:discmgrid}
  Let $\mu$ be the counting measure on $\mathcal{X}=\delta \Z$ with $\delta>0$
and take $\cD=\Delta^+_{\delta}$, the scaled forward difference
  operator $\Delta^+_{\delta}f(x) = \delta^{-1} \left(f(x+\delta) -
    f(x) \right)$.   The inverse $\mathcal{D}^{-1}$ is defined similarly as
  in the previous example. Also, setting $\Delta^{-}_{\delta}f(x) =
  \delta^{-1} \left(f(x) - f(x-\delta) \right)$,  
we have the
discrete product rule 
\begin{equation*}
  \Delta^{+}_{\delta}(f(x)g(x-\delta))     =  g(x) \Delta^{+}_{\delta}f(x)+ f(x)
  \Delta^{-}_{\delta}g(x)
\end{equation*}  
for all $f, g \in \Z^{\star}$ and all $x \in \R$.  Hence Assumption
\ref{ass:prodrule} is satisfied with $\mathcal{D}^{\star} =
\Delta^{-}_{\delta}$   and $l=-\delta$.
\end{example}

\begin{example}[Standard normal] \label{N01} Let  $\varphi$ be the standard normal
  density function 
   so that $\varphi'(x) = - x \varphi(x)$. Let $\mu(x)$ be the standard
  normal measure on $\R$ 
  and  take $\mathcal{D} = \mathcal{D}_{\varphi}$ the differential operator defined by
$$\cD_{\varphi} f (x) = f'(x) - x f(x) = \frac{(f(x)
  \varphi(x))'}{\varphi(x)},$$
see e.g. \cite{ledoux2014stein}. 
Then 
\begin{equation*}
  \mathcal{D}_{\varphi}^{-1}f(x) =
  \frac{1}{\varphi(x)}\int_{\bullet}^x f(y) \varphi(y)dy. 
\end{equation*}
Also we have  the product rule 
\begin{eqnarray*}
\cD_{\varphi}(gf)(x) &=& (gf)'(x)  - x g(x) f(x) \\
&=& g(x) \cD_{\varphi} f(x) + f(x) g'(x).
\end{eqnarray*} 
Hence  Assumption \ref{ass:prodrule} is satisfied  with $\cD^{\star} g =
g'$ and $l=0$.
\end{example}

\begin{example}[Poisson] \label{PoiL} Let $\gamma_{\lambda}$ be the 
Poisson    probability mass function with parameter $\lambda$.  
Let $\mu(x)$ be the corresponding
  Poisson measure on $\Z^+$ and take $\mathcal{D} =
  \Delta^+_{\lambda}$ the difference 
  operator defined by 
  \begin{equation*}
    \Delta^+_{\lambda} f(x) = \lambda f(x+1)-xf(x) = \frac{\Delta^+(f(x)x
    \gamma_{\lambda}(x))}{\gamma_{\lambda}(x)}.  
  \end{equation*}
Then 
\begin{equation*}
   (\Delta^+_{\lambda})^{-1}f(x) = \frac{1}{x
     \gamma_{\lambda}(x)}\sum_{k=\bullet}^{x-1} f(k) \gamma_{\lambda}(k)
\end{equation*}
which
{is} ill-defined at
$x=0$  (see, e.g., \cite{BC05,BaHoJa92}).
We have the product rule 
\begin{equation*}
     \Delta^+_{\lambda}(g(x-1)f(x) ) =g(x) \Delta^+_{\lambda}f(x)
     +  f(x) x   \Delta^-g(x). 
\end{equation*}
Hence  Assumption \ref{ass:prodrule} is satisfied  with $\cD^{\star} g (x) =
x \Delta^-g(x)$ and $l=-1$.
\end{example}

\begin{remark}
  
  In all  examples considered above the choice of $\mathcal{D}$ is, in a sense,
arbitrary and other options are available. In the Lebesgue measure
setting of Example \ref{ex:cpont}
one could, for instance,   use
 $\cD$ the derivative in the sense of distributions, or even 
$    \cD f(x) = \frac{\partial}{\partial t} f(P_tx)$
for $x \mapsto P_tx$ some transformation of $\mathcal{X}$; see
e.g. \cite{ley2013parametric}. In the counting measure setting of Example
\ref{ex:discm} the roles of backward and forward difference operators
can be exchanged; these operators can also be replaced by linear
combinations as, e.g., in \cite{hillion2011natural}.  The discrete
construction is also easily extended to general spacing $\delta\neq1$
: if $\mathcal{X} =\delta \Z$, then we can take $\mathcal{D} =
\Delta_{\delta}^{+}$ such that $\mathcal{D}f(x) = f(x+\delta)-f(x)$.
In the Poisson example one could also consider
\begin{equation*}
  \mathcal{D}f(x) = \frac{\lambda}{x+1} f(x+1)-f(x) =
  \frac{\Delta^+(f(x)    \gamma_{\lambda}(x))}{\gamma_{\lambda}(x)}.
  \end{equation*}
  In all cases less conventional choices of $\mathcal{D}$ can be
  envisaged (even forward differences in the continuous setting). 
\end{remark}
\begin{remark}\label{rem:dim}
  Nowhere is the restriction
  to dimension 1 necessary in this subsection. The need for this
  assumption will become apparent when we use the
  setup to construct a general version of Stein's method. Indeed
  although  our approach should in principle be able to provide
  useful insight into Stein's method for multivariate distributions,
  the method does not fare well in higher dimensions (this fact is
  well-known, see e.g.   \cite{ChMe08,nourdin2013entropy,MR2573554}) and
  we will not discuss multivariate extensions further in this paper.
\end{remark}

\subsection{Canonical Stein class and operator}
\label{sec:operator}

Following \cite{goldstein2013stein} we say that a
subset $\III \subset \cX$ is a finite interval if $\III=\{a,b\} \cap
\cX$ for $a, b \in \R$ with $ a \le b$, and an infinite interval if
either $\III=(-\infty, b\}\cap \cX$ or $\III=\{a, \infty) \cap \cX$ or
$\III= \cX$ (provided, of course, that $\cX$ itself has infinite
length). Here $\{$ is used as shorthand for either $($ or $[$, and similarly $\}$ is either
$)$ or $]$. In the sequel we consistently denote  intervals by $\mathcal{I} = \left\{
  a, b \right\}$ where $-\infty \le a \le b \le +\infty$ (we omit the
intersection with $\mathcal{X}$ unless necessary). 

Now consider a real-valued random variable $X$ on $\cX$ such that
$P_X(A) = {\rm P}(X \in A)$ for $A\in\mathcal{B}$ is absolutely
continuous w.r.t.~$\mu$. Let $p = dP_X/d\mu$ be the Radon-Nikodym
derivative of $P_X$; throughout we call $p$ the \emph{density} of $X$
(even if $X$ is {not a continuous} 
random variable). In the sequel, we only
consider random variables such that $p \in dom(\mathcal{D}) $ and
whose support $supp(p) = \left\{ x \in \mathcal{X} \, | \, p(x)>0
\right\} = : \mathcal{I}$ is an interval of $\mathcal{X}$.  For any
real-valued function $h$ on $\mathcal{X}$ we write
\begin{equation*}
  \E_ph = \E h(X) = \int_{\mathcal{X}}h p d\mu = \int_\III h p d\mu; 
\end{equation*}
this expectation exists for all 
functions $h: \cX \rightarrow \R $ such that
$\E_p |h|<\infty$; we denote this set of functions by $L^1_{\mu}(p) 
\equiv L^1_{\mu}(X)$. 
\begin{definition}\label{def:canonev}
  The \emph{canonical $\mathcal{D}$-Stein class} $\mathcal{F}(p)
  \equiv \mathcal{F}(X) (= \mathcal{F}_{\mu}(p))$ for $X$ is the collection
  of functions   $f \in L^1_{\mu}(p)$ such that (i) $fp \in
 dom(\mathcal{D})$, (ii) $\mathcal{D}(fp) \in L^1(\mu)$ and (iii) $  \int_{\mathcal{I}} \mathcal{D}(fp) d\mu = 0. $ 
  The \emph{canonical $\mathcal{D}$-Stein operator} $\cT_p
 \equiv \cT_X$ for $p$ is the linear operator
  on $\mathcal{F}(X)$ defined as
 \begin{equation}
    \label{eq:3}
    \cT_X f : \mathcal{F}(X) \to L^1_{\mu}(p) : f \mapsto \frac{\cD (fp) }{p},
  \end{equation}
with the convention that  $\cT_X f = 0$ outside of $\III$. We call the
construction $(\mathcal{T}_X,\mathcal{F}(X)) = (\mathcal{T}_p,\mathcal{F}(p))$ a $\mathcal{D}$-Stein
pair for $X$. 
\end{definition}

\begin{remark}
  In  the sequel we shall generally drop the reference to the
  dominating differential $\mathcal{D}$. 
\end{remark}

To avoid triviality we from hereon assume that $\mathcal{F}(X)
\setminus \left\{ 0 \right\}\neq \emptyset$.  Note that
$\mathcal{F}(X)$ is closed under multiplication by constants.
By definition,  $\mathcal{T}_Xf \in L^1_{\mu}(p)$ for all
$f\in\mathcal{F}(X)$, and 
$$
\E[\mathcal{T}_Xf(X)] = \int_\III  \frac{\cD (fp)(x) }{p(x) } p(x) d\mu(x) = \int_\III \cD (fp)(x) d\mu(x) =0,
$$
so that $\mathcal{T}_X$ satisfies {E}quation \eqref{steingen}, 
qualifying it as a Stein operator.

\begin{remark}\label{rem:definop}
The assumption for $\cF(X)$ that $  \int_{\mathcal{I}} \mathcal{D}(fp)
d\mu = 0 $ is made for convenience of calculation but it is not
essential. Indeed sometimes it may be more natural not to impose this
restriction. For example if $\mu$ is the continuous uniform measure on
$[0,1]$ and $p=1$, with $\cD$ the usual derivative,  then imposing
that $\int_0^1 f'(x) dx=f(1) - f(0) = 0$ may not be natural. The price
to pay for relaxing the assumption is that in the definition of
$\cT_X f(X)$ we  would have to subtract this integral, as in
\cite{Stein2004},  to assure that
$\E[\mathcal{T}_Xf(X)] =0$.
\end{remark}

The canonical Stein operator \eqref{eq:3} bears an intuitive
interpretation in terms of the linear operator $\mathcal{D}$.
\begin{proposition}\label{prop:skewadjoint}
For all $f \in \mathcal{F}(X)$  define  the class of functions
\begin{align} \label{dom1}
dom(\cD, X,f) &= \left\{ g \in dom (\cD^{\star}): g(\cdot+l)f(\cdot) \in
  \mathcal{F}(X),  \right. \nonumber \\
& \quad \quad \left.  \EE | f(X) \cD^{\star} (g)(X) | <
  \infty \mbox{ or } \EE | g(X) \mathcal{T}_Xf(X) | < \infty \right\}.
\end{align}
Then 
\begin{equation}\label{eq:basicstein} 
\EE \left[ f(X)\cD^{\star} (g)(X)\right] = -\EE \left[ g(X) \mathcal{T}_Xf(X) \right]
\end{equation}
for all $f \in   \mathcal{F}(X) $  and all $g \in dom(\cD, X, f)$.
\end{proposition}
\begin{proof}
Assumption \ref{ass:prodrule} assures us that 
\begin{equation*}
  \mathcal{D}(g(\cdot +  l)f(\cdot)p(\cdot))(x) = g(x) \mathcal{D}(fp)(x) + f(x) p(x) \mathcal{D}^{\star}g(x)
\end{equation*}
for all   $f \in \mathcal{F}(X)$ and all $g \in dom(\cD^{\star})$. If moreover $g \in
dom(\cD, X, f)$  then $\int_{\mathcal{X}}\mathcal{D}(g(x +
l)f(x)p(x)) d\mu(x) = 0$ and 
\begin{align*}
\E \left[  g(X)  \frac{\cD (fp) }{p} (X) \right]  &= \int_{\III} g \cD(fp) d\mu \\
&= - \int_{\III}fp  \cD^{\star}(g) d \mu \\
&=- \E[  f(X)\cD^{\star}(g)(X)],
  \end{align*}
with both integrals being finite. This yields  \eqref{eq:basicstein}.  
\end{proof}


As anticipated in the Introduction, relationship  \eqref{eq:basicstein} shows that if $\mathcal{D}$ is
skew-adjoint with respect to $\mathcal{D}^{\star}$ under integration
in $\mu$  then the canonical Stein operator is skew-adjoint to
$\mathcal{D}^{\star}$ under integration in the measure $P_X$. This
motivates  the use of the terminology  ``canonical'' in Definition \ref{def:canonev}; we will further
elaborate on 
this topic in Section \ref{sec:stein-char}.

\begin{example}[Example \ref{ex:cpont}, continued] \label{ex:cpontCONT}
 Let $X$ be a random variable with absolutely continuous density
  $p$ with support $\mathcal{I} = \left\{ a, b \right\}$.  Then $\mathcal{F}(X)$ is the collection of
  functions $f:\R\to\R$ such that $fp\in W^{1,1}$ the Sobolev space of
  order 1 on $L^1(dx)$ and 
$   \lim_{x\searrow  a}f(x)p(x)=\lim_{x\nearrow b}f(x)p(x);$ the canonical Stein operator is 
\begin{equation*}
  \mathcal{T}_Xf = \frac{(fp)'}{p}
  \end{equation*}
which we set to 0 outside of $\III$. 
Also,  for
    $f \in \mathcal{F}(X)$, $dom( \left( \cdot \right)', X, f)$ is the
    class of  differentiable functions $g
:\R\to\R$  such that $\int  \left( gfp \right)' dx =0$, $\int | g'fp | dx
  <\infty $ or $\int | g(fp)' | dx <\infty.$
(Note that the first requirement implicitly requires $\int |\left( gfp
\right)'| dx <\infty$.)  In particular all constant functions  are in
$dom(\left( \cdot \right)', X, f)$.  

In the case that $p$ itself is differentiable  (and not only the function $x \mapsto
 f(x)p(x)$ is) we can write 
\begin{equation}
  \label{eq:11}
  \mathcal{T}_Xf (x) = \left(  f'(x) + f(x) \frac{p'(x)}{p(x)} \right)
\mathbb{I}(x \in \III),
\end{equation}
with $\mathbb{I}(\cdot)$ the usual indicator function. This is
operator \eqref{eq:36} from Stein's density approach. 
Note that, in many cases, the constant functions may not belong to
$\mathcal{F}(X)$. 
Operator \eqref{eq:11} was also discussed (under
slightly different -- more restrictive -- assumptions) in
\cite{CS11}. See also \cite{LS12a} for a similar
construction. 
\end{example}
\begin{example}[Example \ref{ex:discm}, continued]  \label{ex:discmCONT}
  Recall $\cD = \Delta^+$ and consider $X$ some discrete random
  variable whose density $p$ has interval support $ \mathcal{I} = 
[a, b]$ (with, for simplicity,
  $a>-\infty$). The {associated}  (forward) Stein operator is
  \begin{equation*}
    \mathcal{T}_Xf = \frac{\Delta^+(fp)}{p},
  \end{equation*}
which we set to 0 outside of $\III$. 
We divide the example
  in two parts. 
  \begin{enumerate}
  \item If $b<+\infty$ :   the  {associated} (forward) canonical Stein class 
    $\mathcal{F}(X)$ is the collection of functions $f :\Z \to \R$
    such that $f(a) = 0$, and,  for
    $f \in \mathcal{F}(X)$, $dom(\Delta^+, X, f)$ is the
    collection of functions $g : \Z \to \R$.
\item If $b = +\infty$ : the (forward) canonical Stein class 
    $\mathcal{F}(X)$ is the collection of  functions $f:\Z\to\R$
    such that $f(a) = 0$ and $\sum_{n=a}^{\infty} |f(n)|p(n)<+\infty$, and  for
    $f \in \mathcal{F}(X)$, $dom(\Delta^+, X, f)$ is the
    collection of  functions $g : \Z \to \R$ such that $\lim_{n\to
      \infty} g(n-1) f(n) p(n)= 0$ and, either  $\sum_{k=a}^{\infty} p(k) \left|
    f(k)   \Delta^+g(k)\right| < \infty $ or $ \sum_{k=a}^{\infty} p(k) \left|
      g(k) \mathcal{T}_Xf(k) \right|<\infty.$
In particular all bounded functions $g$ are in $dom(\Delta^+, X, f)$. 
\end{enumerate} 
If $p$ itself is in $\mathcal{F}(X)$ then we have
\begin{equation*}
  \mathcal{T}_Xf(x) = f(x+1) \frac{p(x+1)}{p(x)} - f(x). 
\end{equation*}
Similarly it is straightforward to
define a backward Stein class and operator.
\end{example}

\begin{example}[Example \ref{N01}, continued] 
\label{exN01cont} Let $X$ be a random
  variable with 
  density $p$ with support $\III=\left\{ a, b \right\}$ with respect to $\varphi(x)dx$ the Gaussian
  measure. Recall $\mathcal{D}_{\varphi}f(x) = f'(x) - xf(x)$ and
  $\mathcal{D}^{\star}g(x) = g'(x)$. 
Then $\mathcal{F}(X)$ is the collection of
  functions $f:\R\to\R$ such that $fp \in L^1(\varphi)$ is absolutely
  continuous, 
  $\int_{\R}| \mathcal{D}_{\varphi}(fp)|\varphi(x)dx < \infty $  and 
$   \lim_{x\searrow  a}f(x)p(x)\varphi(x)=\lim_{x\nearrow
  b}f(x)p(x)\varphi(x);$
the canonical Stein operator is 
\begin{equation*}
  \mathcal{T}_Xf = \frac{\mathcal{D}_{\varphi}(fp)}{p} = \frac{(fp\varphi)'}{p\varphi}
  \end{equation*} 
which we set to 0 outside of $\III$. 
Also,  for
    $f \in \mathcal{F}(X)$, $dom(\mathcal{D}_{\varphi}, X, f)$
    contains all differentiable  functions $g
:\R\to\R$  such that $gf \in L^1_{\mu}(p)$ (or, equivalently, $gfp \in
L^1(\varphi)$), $\int  \left( gfp \varphi \right)' dx =0$, and
either $\int | g'f |
  p\varphi dx
  <\infty $ or $\int | g(fp \varphi)' | dx <\infty$.
In particular all constant functions  are in
$dom(\mathcal{D}_{\varphi},  X, f)$.  
 {The}  above construction {can also be obtained} directly  by
replacing $p$ with $p \varphi$ in  Example \ref{ex:cpontCONT}.
\end{example}

\begin{example}[Example \ref{PoiL}, continued] \label{exPoiLcont}
  Recall $\mathcal{D} = \Delta^+_{\lambda}$ and consider $X$ some
  discrete random variable whose density $p$ has interval support $
  \mathcal{I} = [0, b]$. The
  (forward) Stein operator is
  \begin{equation*}
    \mathcal{T}_Xf = \frac{\Delta^+_{\lambda}(fp)}{p} =
      \frac{\Delta^+(f(x) x p(x) \gamma_{\lambda}(x))}{p(x) \gamma_{\lambda}(x)}, 
  \end{equation*}
which we set to 0 outside of $\III$. Then, as in the previous example,
we simply recover the construction of Example \ref{ex:discmCONT} with $f(x)$
replaced by $xf(x)$ (and thus no condition on $f(0)$) and $p(x)$
replaced by $p(x) \gamma_{\lambda}(x)$. 
\end{example}

\begin{remark}
  As noted already in the classic paper \cite{DZ91}, the abstract
  theory of Stein operators is closely connected to Sturm-Liouville
  theory. This connection is quite easy to see from our notations and
  framework; it remains however outside of the scope of the present
  paper and will be explored in future publications.
\end{remark}

\subsection{The canonical inverse Stein
  operator}\label{sec:canon-inverse-stein}

The Stein operator being defined (in terms of $\mathcal{D}$), we now
define its inverse (in terms of $\mathcal{D}^{-1}$). To this end first
note that if $\mathcal{D}(fp) = hp$ for $f \in \mathcal{F}(X)$ then
$\mathcal{T}_X(f) = h$. As $\mathcal{D}(fp+\chi) = hp$  for any $\chi \in
ker(\mathcal{D})$, to define a unique right-inverse of $\mathcal{T}_X$
we make  the following assumption.
\begin{assumption}\label{ass:ker}
  $ker(\mathcal{D}) \cap L^1(\mu) = \left\{ 0 \right\}$. 
\end{assumption}

 This assumption ensures that the only $\mu$-integrable $\chi$ is 0 and
thus $\mathcal{T}_X$  (as an operator acting on $\mathcal{F}(X)\subset
L^1(\mu)$) possesses a \emph{bona fide} inverse, and also that
$ker(\mathcal{D}) \cap L^1_{\mu}(p) = \left\{ 0 \right\}$.

\begin{definition}\label{def:invstop}
Let $X$ have density $p$ with support $\III$.  
The \emph{canonical   inverse Stein operator}
$ \cT_p^{-1} \equiv  \cT_X^{-1}$ for $X$ is defined for all $h$ such
that $hp \in im(\mathcal{D})$ as the unique function $f \in
\mathcal{F}(X)$ such that $\mathcal{D}(fp) = hp$. 
\end{definition}
We will use the shorthand 
\begin{equation*}
  \cT_X^{-1} h  = \frac{ {\cD}^{-1} (h p)}{p}
\end{equation*} 
with the convention that $\cT_X^{-1} h = 0$ outside of $\III$.

 We state the counterpart of Proposition 
 \ref{prop:skewadjoint} for the
inverse Stein operator. 
\begin{proposition}\label{prop:invsteiskew}
 Define the class of functions
 \begin{equation*}
 \mathcal{F}^{(0)}(X)  = \{ h \in
 im(\mathcal{T}_X) : hp = \mathcal{D}(fp)  \mbox{ with } f \in
 \mathcal{F}(X) \}.  
 \end{equation*}
 Then 
 \begin{equation}\label{eq:8}
   \EE \left[\mathcal{T}_X^{-1}h (X)  \mathcal{D}^{\star}g (X)
   \right] = - \EE \left[ g (X)  h (X)  \right] 
 \end{equation}
 for all  $h \in
 \mathcal{F}^{(0)}(X)$ and all  $g \in
 dom(\mathcal{D}, X,  \mathcal{T}_X^{-1}h )$.  
  \end{proposition}

\begin{example}[Example \ref{ex:cpont},
  continued] \label{ex:cpontcontcont}
 Let $X$ have
  support $\III = \left\{ a, b \right\}$ with Stein class
  $\mathcal{F}(X)$ and Stein operator $\mathcal{T}_X(f) = (fp)'/p$.
  Then 
\begin{equation*}
  \mathcal{T}_X^{-1}h(x) = \frac{1}{p(x)}\int_{a}^x h(u) p(u) du = -
  \frac{1}{p(x)} \int_x^b h(u) p(u) du
\end{equation*}
for all $h\in \mathcal{F}^{(0)}(X)$ the collection of  functions $h\in
L^1_{\mu}(p)$ such that $\EE_ph=0$.
\end{example}

\begin{example}[Example \ref{ex:discm},
  continued]\label{ex:discmcontcont}  
   Let $X$ have
  support $\III = [ a, b ]$  with Stein class
  $\mathcal{F}(X)$ and Stein operator $\mathcal{T}_X(f) = \Delta^+(fp)/p$.
  Then
\begin{equation*}
  \mathcal{T}_X^{-1}h(x) = \frac{1}{p(x)}\sum_{k=a}^x h(k) p(k) = -
  \frac{1}{p(x)} \sum_{k=x+1}^b h(k) p(k)
\end{equation*}
for all $h\in \mathcal{F}^{(0)}(X)$  the collection of functions $h$
  such that $\EE_ph=0$. 
\end{example}
The inverse operator and corresponding sets in Example~\ref{N01} (resp., 
Example~\ref{PoiL}) are simply obtained by replacing $p$ with
$\varphi p$ (resp., with $\gamma_{\lambda} p$) in
Example~\ref{ex:cpontcontcont} (resp., in Example
\ref{ex:discmcontcont}).

\subsection{Stein  differentiation and the product rule}
\label{sec:stein-oper-stein}

 Define the new class of functions
\begin{equation*}
  dom(\cD, X):=\bigcap_{f\in\mathcal{F}(X)}dom(\cD,X,f)
\end{equation*}
with $dom(\cD,X,f)$ as in \eqref{dom1}. Then the following holds. 
\begin{lemma}\label{lem:constantsd}
  If the constant function $1$ belongs to $dom (\cD) \cap dom
  (\cD^{\star})$, then all constant functions are in
  $ker(\mathcal{D}^{\star})$ and in $dom(\mathcal{D}, X)$.
\end{lemma}
\begin{proof}
 Taking $g\equiv 1$
 in \eqref{eq:2} we see that 
$  \mathcal{D}f(x) = \mathcal{D}f(x) + f(x) \mathcal{D}^{\star} 1(x)$
for all $f \in dom(\mathcal{D})$. Taking $f \equiv 1$ ensures the first
claim.  The second claim then follows immediately. 
\end{proof}

From here onwards we make the following assumption. 

\begin{assumption}\label{ass:1ind}
  $1 \in dom(\mathcal{D})\cap dom(\mathcal{D}^{\star})$.
\end{assumption}

Starting from the product rule \eqref{eq:2} we also obtain the
following  differentiation rules for  $\mathcal{D}$ and
$\mathcal{D}^{\star}$. 

\begin{lemma}
  \label{lem:DDDiff}
Under  Assumptions \ref{ass:prodrule}  and \ref{ass:1ind} we have 
\begin{enumerate}
\item \label{vv1} $\mathcal{D}g(\cdot + l) = g \mathcal{D}1 + \mathcal{D}^{\star}g$
\item \label{vv2} $\mathcal{D}^{\star}(fg) = g \mathcal{D}^{\star}f
+ f( \cdot + l) \mathcal{D}^{\star} g$
\end{enumerate}
for all $f, g \in dom(\mathcal{D}) \cap dom(\mathcal{D}^{\star})$. 
\end{lemma}
\begin{proof}
 Claim \ref{vv1}. is immediate. To see \ref{vv2}.,  using  Assumption \ref{ass:prodrule} we write 
  \begin{align*}
    \mathcal{D}^{\star}(f g) & = - fg
    \mathcal{D}1 + 
    \mathcal{D}(f(\cdot+l)g(\cdot+l))  \\
    & = - f g \mathcal{D}1 + g \mathcal{D}f(\cdot+l) + f(\cdot+l) \mathcal{D}^{\star}g.
  \end{align*}
Applying Claim \ref{vv1}. to the second summand we then get 
\begin{align*}
    \mathcal{D}^{\star}(fg) & =  - fg \mathcal{D}1 + fg \mathcal{D}1
    + g \mathcal{D}^{\star}f
+ f( \cdot + l) \mathcal{D}^{\star} g \\
& = g \mathcal{D}^{\star}f
+ f( \cdot + l) \mathcal{D}^{\star} g.
\end{align*}
\end{proof}
\begin{remark}
  From Point \ref{vv1}. in Lemma \ref{lem:DDDiff} we see that if   $l=0$ and $1 \in
ker(\mathcal{D})$  then  $\mathcal{D} =
\mathcal{D}^{\star}$ on $dom(\cD) \cap dom(\cD^*)$. Neither of these assumptions are always
satisfied (see Examples \ref{ex:discm} and \ref{N01}).
\end{remark}


The following result is the basis of {what we call}
``Stein
differentiation''. It is also the key to the standardizations leading
to the different Stein operators that will be discussed in Section 
 \ref{sec:examples}. 

\begin{theorem}[Stein product rule]
  \label{sec:steindiff}
The Stein triple  $(\mathcal{T}_X, \mathcal{F}(X), dom(\mathcal{D}, X,
\cdot))$ satisfies  the product rule
\begin{equation} \label{steindiff}
 f(x)  \cD^{\star} (g)(x)  + g(x) \mathcal{T}_Xf(x) = \mathcal{T}_X (f(\cdot) g(\cdot + l) )(x)
\end{equation}
for $f \in \mathcal{F}(X)$ and $g \in dom(\mathcal{D}, X, f)$.
\end{theorem}
\begin{proof}
Use  Assumption \ref{ass:prodrule} to deduce 
\begin{eqnarray*}
 f(x)  \cD^{\star} (g)(x)  + g(x) \mathcal{T}_Xf(x) &=&  f(x)  \cD^{\star} (g)(x)  + g(x)  \frac{ \cD( f p)(x)}{p(x)} \nonumber \\
&=& \frac{1}{p(x)} \cD(f(\cdot) p(\cdot) g(\cdot + l) )(x),
\end{eqnarray*}
which is the claim.
\end{proof}
 To see how \eqref{steindiff} can be put to use,
let $h \in L^1_{\mu}(X)$ and consider the equation
\begin{equation}\label{eq:genstequationfg} 
  h(x) - \E h(X) = f(x)  \cD^{\star} (g)(x)  + g(x) \mathcal{T}_Xf(x), \quad x \in \III.
\end{equation}
As discussed in the Introduction,  {E}quation \eqref{eq:genstequationfg}
 is indeed a \emph{Stein equation for the target $X$} in
the sense of \eqref{steineq}, although the solutions of
\eqref{eq:genstequationfg} are now pairs of functions $(f,g)$ with
$f\in\mathcal{F}(X)$ and $g \in dom(\cD,X,f)$ which satisfy the relationship
\begin{equation}
  \label{eq:steqsolsol}
  f (\cdot) g ( \cdot + l) = \mathcal{T}_X^{-1}(h - \EE_ph).
\end{equation}
We stress that although $fg$ is uniquely defined by
\eqref{eq:steqsolsol}, the individual 
$f$ and $g$ are not (just consider multiplication by constants).

Equation \eqref{eq:genstequationfg} and its solutions
\eqref{eq:steqsolsol} are not equivalent to  {E}quation
\eqref{steineq} and its solutions  already available from
the literature, but rather contain them, as illustrated in the
following example.

\begin{example}[Example \ref{ex:cpont}, continued] Taking $g = 1$
  (this is always permitted by Lemma~\ref{lem:constantsd}) and $p$
  differentiable we get the equation
\begin{equation}
  \label{eq:40}
  h(x) - \E h(X) = f'(x) + \frac{p'(x)}{p(x)} f(x), \quad x \in \III,
\end{equation}
whose solution is to be some function $f \in \mathcal{F}(X)$, as in
e.g. \cite{LS12a}. If the constant function $f \equiv 1$ is in $\cF(X)$ then 
keeping instead $g$ variable but taking $f\equiv 1$  yields the equation 
\begin{equation}
\label{eq:42}   h(x) - \E h(X) = g'(x) + \frac{p'(x)}{p(x)} g(x), \quad x \in \III, 
\end{equation}
whose solution is any function in $dom(\cD, X,1)$ the collection of
functions $g \in \mathcal{F}(X)$ such that $gp'/p \in L^1_{\mu}(X)$, a
family of equations considered e.g. in \cite{Stein2004}. Similar
considerations hold in the settings of examples \ref{ex:discm},
\ref{N01} and \ref{PoiL}.  We stress the
fact that the difference between \eqref{eq:40} and \eqref{eq:42} lies
in the space of solutions. 
 \end{example}


\subsection{Stein characterizations}
\label{sec:stein-char}

Pursuing the tradition in the literature on Stein's method, we provide
a general family of \emph{Stein characterizations for $X$}.   Aside from
Assumptions \ref{ass:prodrule}, \ref{ass:ker} and \ref{ass:1ind} we
will further need the following two assumptions to hold.
\begin{assumption}
  \label{ass:constantsonlyinker}
$f \in ker(\mathcal{D}^{\star})$ if $f \equiv \alpha$ for some $\alpha
\in \R$.  
\end{assumption}
\begin{assumption}
  \label{ass:uniquescore}
$\mathcal{D}f/f = \mathcal{D}g/g$  for $f, g \in dom
(\mathcal{D})$ if and only if  $f / g \equiv  \alpha$ 
for some $\alpha \in \R$. 
\end{assumption}
Both assumptions are simultaneously satisfied in all examples
discussed in Section \ref{sec:gener-vers-steins}.


\begin{theorem}\label{th:generalsetting} 
  Let $Y$ be a random element with the same support as $X$ and assume
  that the law of $Y$ is absolutely continuous \mbox{w.r.t.} $\mu$ with
  Radon-Nikodym derivative $q$.
  \begin{enumerate}
  \item Suppose that $\mathcal{F}(X)$ is dense in $L^1_{\mu}(p)$ and that $\frac{q}{p} \in
dom(\cD^*)$.  Take $g
    \in dom(\cD, X)$ which is $X$-a.s. never 0 and assume that $g
    \frac{q}{p} \in dom(\cD, X)$.  Then
  \begin{equation}\label{eq:3a}
     Y  \stackrel{\mathcal{D}}{=} X  \mbox{ if and only if }  \E
     \left[ f(Y)  \cD^{\star} (g) (Y) \right] = -\E \left[ g(Y)
       \mathcal{T}_Xf(Y) \right] 
  \end{equation}
for all $f \in \mathcal{F}(X) $.
\item  Let $f \in  \mathcal{F}(X)$ be $X$-a.s. never zero and assume that
  $dom(\cD,X, f)$  is dense in $L^1_{\mu}(p)$. 
 Then  
 \begin{equation}
\label{eq:3b}
    Y \stackrel{\mathcal{D}}{=} X  \mbox{ if and
      only if }  \E \left[  f(Y)  \cD^{\star} (g) (Y) \right] = -\E \left[ g(Y)
      \mathcal{T}_Xf(Y) \right]
 \end{equation}
for all $g \in dom(\mathcal{D}, X,f)
$.
\end{enumerate}
\end{theorem}
\begin{remark}
The assumptions leading to \eqref{eq:3a} and \eqref{eq:3b} can  be relaxed
by removing the assumption that $Y$ and $X$ share a support $\III$ but instead 
conditioning on the event that $Y \in \III$ and  writing 
   $Y \, | \, Y \in \III
\stackrel{\mathcal{D}}{=} X$ to indicate that $p = c q$  on $\III$, for a constant $c=P(Y \in \III)$, see
\cite{LS12a}. 
\end{remark}

 \begin{proof}


   The sufficient conditions are immediate. Indeed,  from
   \eqref{eq:basicstein}, if $Y$ has the same distribution as $X$ then
   \eqref{eq:3a} and \eqref{eq:3b} hold true. 

   We now prove the necessity. We start with statement 1. Let $g$ be
   such that $g q/p \in dom(\cD, X)$.  Then, $g q/p \in
   dom(\mathcal{D}^{\star})$ and,  for all $f\in
   \mathcal{F}(X)$,  we have $\mathcal{D}^{\star}(g q/p) fp \in
   L^1(\mu)$ as well as $\int \mathcal{D} \left( g(\cdot + l)
     \frac{q(\cdot+l)}{p(\cdot+l)}f(\cdot)p(\cdot) \right) d\mu = 0$ and we can
   apply  \eqref{eq:4} to get 
\begin{align*}
  \E \left[  g(Y)
      \mathcal{T}_Xf(Y) \right] 
    = \int g \frac{q }{p } \cD(fp)  d\mu = - \int  f  p  \cD^{\star} \left( g \frac{q}{p}\right)  d\mu. 
\end{align*}
Supposing \eqref{eq:3a} gives 
\begin{equation*}
  \int \cD^{\star} \left( g \frac{q}{p}\right)  f  p  d\mu =  \int
   f \cD^{\star} \left( 
  g \right) q d\mu =  \int
   f \cD^{\star} \left( 
  g \right) \frac{q}{p}p d\mu 
\end{equation*}
for all $f \in \mathcal{F}(X)$. On the one hand, as  $\mathcal{F}(X)$
is assumed to be dense in
$L^1_{\mu}(p)$, it follows that $\cD^{\star} \left( g \frac{q}{p}\right) =  \frac{q}{p}\cD^{\star} \left( g
\right)$ $p-\mbox{a.e.}$ and, 
on the other hand, by Claim 2. in  Lemma \ref{lem:DDDiff} we know that 
$\cD^{\star} \left( g \frac{q}{p}\right)  = \frac{q}{p}
   \mathcal{D}^{\star} g + g(\cdot+l) \mathcal{D}^{\star} \left(
     \frac{q}{p} \right). $
Equating these two expressions gives that 
$ g(\cdot+l) \mathcal{D}^{\star} \left(
  \frac{q}{p} \right) = 0 $ $p-\mbox{a.e.}$ and,  as $g$ is
$p$-a.e.~never 0 we obtain that
$$  \mathcal{D}^{\star} \left(
     \frac{q}{p} \right) = 0 \quad 
 \quad p-\mbox{a.e.}.$$
Assumption \ref{ass:constantsonlyinker} now gives that there is a constant $c$ such that  $p = c q$ except on a set of $p$-measure 0. As both $p$ and $q$ integrate to 1, it must be the case that $c=1$, and so 
$p=q $ on $supp(p)$, which gives the first assertion.

We tackle statement 2.  If $g \frac{q(\cdot-l)}{p(\cdot-l)} \in dom(\mathcal{D}, X, f)$
then 
\begin{equation*}
  \int \mathcal{D}(f(\cdot) \frac{q(\cdot)}{p(\cdot)}g(\cdot+l))
d\mu =\int \mathcal{D}(f(\cdot)g(\cdot +l) q(\cdot))
d\mu =  0
\end{equation*}
 so that 
\begin{align*}
  \E \left[    f(Y) \cD^{\star} \left( g \right)(Y) \right]
 = - \int     g    \cD(fq)   d\mu  
  = - \int       \frac{\cD(fq) }{p } g p  d\mu.  
\end{align*}
Supposing \eqref{eq:3b} gives 
\begin{equation*}
\int       \frac{\cD(fq) }{p } g p  d\mu   =  \int
\frac{\cD(fp)}{p} g q d\mu   =  \int
\frac{\cD(fp)}{p} g \frac{q}{p}p d\mu
\end{equation*}
for all $g \in dom (\mathcal{D}, X, f)$. As 
$dom(\mathcal{D}, X, f)$ is assumed to be dense  in $L^1(\mu)$ it follows that $\cD(fq) =
\cD(fp)  \frac{q}{p}$. On the one hand 
$\mathcal{D}(fp) \frac{q}{p}  = f(\cdot -l) \frac{q}{p} \mathcal{D}p +
q \mathcal{D}^{\star}f(\cdot-l)
$
and, on the other hand, 
$  \mathcal{D}(fq) = f( \cdot - l) \mathcal{D}(q) + q
\mathcal{D}^{\star}f(\cdot-l).$
Simplifying and using the fact that $f$ is never 0 we deduce the
equivalent score-like condition 
\begin{align*}
  \frac{\cD(q)}{q} =
\frac{\cD(p) }{p}  \quad p-a.e.
\end{align*}
Assumption \ref{ass:uniquescore} gives the conclusion. 
\end{proof}
Theorem \ref{th:generalsetting} generalizes   the literature on this
topic in a subtle, yet fundamental, fashion. To see this first take
$g\equiv1$ in \eqref{eq:3a} (recall that this is always permitted) to
obtain the Stein characterization
  \begin{equation*}
    Y \stackrel{\mathcal{D}}{=} X  \mbox{ if and
      only if } \E \left[ \mathcal{T}_Xf(Y) \right]=0  \mbox{ for
    all }f\in\mathcal{F}(X)
  \end{equation*}
  which is valid as soon as the densities of $X$ and $Y$ have same
  support and $q/p \in dom (\mathcal{D}, X,\cdot)$. This is the
  characterization given in \cite{LS12a,LS13b}. If $f \equiv 1$ is in
  $\cF(X)$
then, for this choice of $f$ in \eqref{eq:3b} we obtain
  the Stein characterization
  \begin{equation*}
    Y \stackrel{\mathcal{D}}{=} X  \Longleftrightarrow  \E [ g'(Y)] = - \E \left[\frac{p'(Y)}{p(Y)}g(Y)
    \right]=0  \mbox{ for
    all } g \in dom(\mathcal{D}, X, 1).
  \end{equation*}
Here we assume that
$p$ and $q$  share same
support. The condition $g \in dom (\mathcal{D}, X, 1)$ is equivalent
to $g(\cdot+l) \in \mathcal{F}(X)$ and $\E \left| g(X)
  \mathcal{T}_X1(X) \right|<\infty$.  This is the general
characterization investigated  in \cite{Stein2004}.

 \begin{remark}\label{rem:theconstfunct}
The hypothesis  that the constant function 1 belongs
to $\mathcal{F}(X)$ is not a small
assumption. Indeed, we easily see that 
\begin{equation*}
  1 \in \mathcal{F}(X) \Longleftrightarrow p'/p \in L^1_{\mu}(X) \mbox{ and }\int_{\mathcal{I}}p'(x) dx =
0. 
\end{equation*}
This condition is not satisfied e.g. by the exponential
distribution $p(x) = e^{-x}\mathbb{I}(x \ge 0)$ (because the integral
of the derivative is not 0) nor by the arcsine distribution $p(x) =
1/\sqrt{x(1-x)}\mathbb{I}(0 < x < 1)$ (because
the derivative is not integrable). 
 \end{remark}

\begin{remark}
Our approach is reminiscent of Stein
characterizations of birth-death processes where one can choose the
death rate and adjust the birth rate accordingly, see \cite{Ho04} and
\cite{eichelsbacher2008stein}. 
\end{remark}


\subsection{Connection  with biasing}\label{sec:conn-with-bias}
  In \cite{GR97} the notion of a zero-bias random variable was
  introduced. Let $X$ be a mean zero random variable with finite,
  nonzero variance $\sigma^2$. We say that $X^*$ has the $X$-zero
  biased distribution if for all differentiable $f$ for which $\E X
  f(X)$ exists, $$\E X f(X) = \sigma^2 \E f'(X^*).$$ Furthermore the
  mean zero normal distribution with variance $\sigma^2$ is the unique
  fixed point of the zero-bias transformation.

  More generally, if $X$ is a random variable with density $p_X \in
  {\it{dom}}(\cD^*)$ then for all $f \in {\it{dom(\cD)}}$, by
  \eqref{eq:2} we have
$$ 
  p_X(x) \cT_X(f)(x) =   \mathcal{D}(f(x)p_X(x) ) = p_X(x-l) \mathcal{D}f(x)
 + f(x) \mathcal{D}^{\star}p_X(x-l)
$$ 
and so 
$$ \E \left[ \frac{p_X(X-l) }{p_X(X)}  \mathcal{D}f(X) \right] + \E
\left[ f(X) \frac{ \mathcal{D}^{\star}p_X(X-l)}{p_X(X)} \right]
=0.  $$  
This equation could lead to the definition of a transformation which
maps a random variable $Y$ to $Y^{(X)}$ such that, for all $f \in
{\it{dom}}(\cD)$ for which the expressions exist,
$$  \E  \left[ \frac{p_X(Y^{(X)}-l) }{p_X(Y^{(X)})}
  \mathcal{D}f(Y^{(X)}) \right] = -  \E \left[  f(Y) \frac{
    \mathcal{D}^{\star}p_X(Y-l)}{p_X(Y)}  \right] .$$ 
For some conditions which give the existence of such $Y^*$ see
\cite{GR05}. As an illustration, in the setting of Example 
 \ref{ex:cpont}, if the density $p$ is log-concave (so that  $-p'/p$ is
 increasing) then the existence of the
 coupling $Y^{(X)}$ is straightforward via   the
 Riesz representation theorem, as in \cite{GR97}.

Finally assume that $\cF(X) \cap {\it{dom}}(\cD)$ is
dense in $L^1_{\mu}(X)$.  To see that $Y =_d X$ if and only if $Y^{(X)}=_d
Y$, first note that by construction if $Y =_d X$ then $Y^{(X)}=_d
Y$. On the other hand, if $Y^{(X)}=_d Y$, then $ \E \cT_X(f)(Y) =0$
for all $f \in \cF(X)\cap {\it{dom}}(\cD)$, and the assertion follows
from the density assumption and \eqref{eq:3a}. Hence \eqref{eq:3a} can
be used to establish distributional characterizations based on biasing
equations.


\section{Stein operators}
\label{sec:examples}

Let $X$ be a random variable with support $\mathcal{X}$,  let $\cD$ be
a linear  operator acting  on $\mathcal{X}^{\star}$ and 
satisfying Assumptions \ref{ass:prodrule} and \ref{ass:ker}.  
There are now two seemingly antagonistic points of view : 
\begin{enumerate}[-]
\item  In the
Introduction we mention the fact that Stein's method for $X$ relies
on a pair $(\mathcal{A}_X, \mathcal{F}(\mathcal{A}_X))$ with
$\mathcal{A}_X$ a differential operator acting on
$\mathcal{F}(\mathcal{A}_X)$ a class of functions. For any given $X$,
the literature on Stein's method contains \emph{many} different such
(not necessarily first order!) operators and classes.
\item In Section
\ref{sec:gener-vers-steins}, we claim to obtain ``{the}''
canonical operator associated to $X$, denoted $\mathcal{T}_X$, acting
on ``{the}'' canonical class $\mathcal{F}(X)$ (uniqueness up to the
choice of $\mathcal{D}$) with unique inverse $\mathcal{T}_X^{-1}$.    
\end{enumerate}
In this  section we merge these two points of view. Our
general point of view is that a Stein operator for a random variable
$X$ is any operator that can be written in the form
  \begin{equation}
    \label{eq:32}
    \mathcal{A}_X : \mathcal{F}(X) \times dom(\cD, X, \cdot)  \to
    \mathcal{X}^{\star} : (f, g) \mapsto \mathcal{T}_X(fg), 
  \end{equation}
and, given $h \in L^1_{\mu}(X)$,  the corresponding Stein equation is
\begin{equation*}
  h - \E h(X) =   \mathcal{A}_X(f, g)
\end{equation*}
whose solutions are any functions $f \in \mathcal{F}(X)$ and $g\in
dom(\mathcal{D}, X, f)$ such that $fg = \mathcal{T}_X^{-1}(h - \E
h(X))$. There are many ways to particularise  \eqref{eq:32}, such as 
\begin{enumerate}
\item fix $f\in \mathcal{F}(X)$ and let $g$ vary in
 $dom(\mathcal{D}, X, f)$,
\item fix $g \in dom(\mathcal{D}, X)$ and
 let  $f$ vary in $\mathcal{F}(X)$,
\item let $f$ and $g$ vary
 simultaneously.
\end{enumerate}
We refer to these mechanisms as \emph{standardizations}. 

For the first approach pick a function ${f} \in \mathcal{F}(X)$ and
define the operator
\begin{equation}
  \label{eq:22}
  \mathcal{A}_Xg = \mathcal{T}_X \left( {f}  (\cdot) g (\cdot + l)  \right) = {f}  \cD^{\star}
  (g)  + g \mathcal{T}_X {f} 
\end{equation}
acting on functions $g \in \mathcal{F}(\mathcal{A}_X) = dom(\mathcal{D}, X,
{f} )$. The corresponding Stein equation is
\begin{equation*}
  h -\E h(X) = \mathcal{A}_Xg
\end{equation*}
whose solutions are $g \in dom(\mathcal{D}, X, {f} )$ given by $g =
\mathcal{T}_X^{-1}( h -\E h(X))/{f} $. 

The second option  is to fix a function
  $g \in dom(\mathcal{D}, X)$ and define the operator
  \begin{equation}
    \label{eq:23}
    \mathcal{A}_{X}f = \mathcal{T}_X \left( f (\cdot) g (\cdot + l) \right) = f
    \cD^{\star} (g)  + g \mathcal{T}_Xf
  \end{equation}
  acting on functions $f \in \mathcal{F}(X)$. In this case  solutions 
  of the Stein equation are $f \in \mathcal{F}(X)$ given by $f =
\mathcal{T}_X^{-1}( h -\E h(X))/g$.  

The third option is to consider operators of the form 
\begin{equation}
  \label{eq:9}
  \mathcal{A}_{X}(f,g) = \mathcal{T}_X \left( f (\cdot) g (\cdot + l) \right) = f
    \cD^{\star} (g)  + g \mathcal{T}_Xf
\end{equation}
acting on functions $(f, g) \in \mathcal{G}_1\times \mathcal{G}_2$
where $\mathcal{G}_1, \mathcal{G}_2 \subseteq \mathcal{X}^{\star}$
are  such that $f (\cdot) g (\cdot + l)
\in \mathcal{F}(X)$. For example we could consider $\mathcal{G}_i$
polynomial functions or exponentials and pick $\mathcal{G}_j$ with
$j\neq i$ so as to satisfy the assumptions. Solutions of the Stein
equation are pairs of functions such that $f (\cdot) g (\cdot + l)
=\mathcal{T}_X^{-1}( h -\E h(X))$.

\begin{remark} The use of the notation $c$ in \eqref{eq:22} relates to
  the notation in \cite{goldstein2013stein}, where the idea of using a
  $c$-function to generate a family of Stein operators \eqref{eq:22}
  was first proposed (in a less general setting).
\end{remark}

\begin{remark}
Although appearances might suggest
otherwise, operators \eqref{eq:22} and \eqref{eq:23} are not
necessarily first order differential/difference operators. One readily
obtains higher order operators by considering, for example, classes
$\mathcal{F}_A(X)$ of functions of the form $f = \mathcal{D}^k\tilde{f}$ for
$\tilde{f}$ 
appropriately chosen; see Section~\ref{sec:distr-satisfy-diff}. 
\end{remark}

The difference between \eqref{eq:22}, \eqref{eq:23} and \eqref{eq:9}
is subtle (the first two being particular cases of the third). The
guiding principle is to find a form of Stein equation for which the
solutions are smooth. The remainder of the Section is dedicated to
illustrating standardizations under several general assumptions on the
target density, hereby providing interesting and important families of
Stein operators.


\subsection{Stein operators via score functions}
\label{sec:score-funct-stein}

Suppose that $X$ is such that the constant function $1 \in
\mathcal{F}(X)$  and define 
\begin{equation}
  \label{eq:10}
u(x) = \mathcal{T}_X1(x) =  \frac{\cD p(x)}{p(x)}
\end{equation}
 the so-called score function of $X$.  Then taking
${f} = 1$ in \eqref{eq:22} we introduce the operator
\begin{equation} \label{scoreeq}
 \mathcal{A}_Xg(x) = \cD^\star g(x-l) + u(x) g(x-l) 
\end{equation}
acting on  
$
  \mathcal{F}(\mathcal{A}_{X}) =  dom(\mathcal{D},
X, 1).
$
The corresponding Stein equation is
\begin{equation*}
  \bar{h}(u) = \cD^\star g(x-l) + g(x-l) u(x) 
\end{equation*}
for $\bar{h}$ any function with $X$-mean zero;  solutions of this 
equation are the functions  
\begin{equation*}
  g_h=  \mathcal{T}_X^{-1} \left( \bar{h} \right)
\end{equation*}
and bounds on these functions (as well as on their derivatives) are
crucial to the applicability of Part B of Stein's method through
operator \eqref{scoreeq}.

In the continuous setting of Example \ref{ex:cpont} we recover
operator \eqref{eq:36}.  In this case  $\mathcal{F}(\mathcal{A}_X)$ is
the set of all differentiable
functions $g$ such that 
\begin{equation*}
  \E \left|g'(X)  \right| <\infty \mbox{ and } \E \left| g(X) u(X)
  \right|<\infty. 
\end{equation*}
These are the  conditions (27) and (28) from \cite[Proposition
4]{Stein2004}.  

\begin{remark}
  The terminology ``score function'' for the function
  $\mathcal{D}p(x)/p(x)$ is standard (at least in the continuous
  case); it is inherited from the statistical literature.
\end{remark}

\subsection{Stein operators via the Stein kernel}
\label{sec:stein-oper-stein-1}

Suppose that   $X$ has finite mean $\nu$ and define 
\begin{equation}
  \label{eq:15}
  \tau(x) = \mathcal{T}_X^{-1}(\nu-Id) 
\end{equation}
a function which we call the \emph{Stein kernel} of $X$ (see
forthcoming Remark~\ref{rem:steke} as well as
Sections~\ref{sec:comp-stein-kern} and \ref{sec:sum}).  Next take ${f} = \tau $ in
\eqref{eq:22} (this is always permitted) and introduce the operator
\begin{equation} \label{eq:tau}
\mathcal{A}_Xg(x) =  \tau(x) \cD^\star g(x-l) + (\nu-x) g(x-l)
\end{equation}
acting on 
$
  \mathcal{F}(\mathcal{A}_X) = dom(\mathcal{D}, X, \tau).
$
The corresponding Stein equation is
\begin{equation*}
  \bar{h}(x) =  \tau(x) \cD^\star g(x-l) + (\nu-x) g(x-l)
\end{equation*}
for $\bar{h}$ any function with $X$-mean 0; solutions of this
equation are the functions 
\begin{equation*}
  g_h = \frac{1}{\tau} \mathcal{T}_X^{-1}(\bar{h})
\end{equation*}
and bounds on these functions (as well as on their derivatives) are
crucial to the applicability of Part B of Stein's method via operator
\eqref{eq:tau}. 

In the continuous setting of Example \ref{ex:cpont},
$\mathcal{F}(\mathcal{A}_X)$ is the set of all differentiable
functions such that
\begin{equation*}
  \E \left|g(X) (X-\nu)  \right|<\infty \mbox{ and } \E \left| g'(X)
    \tau(X)\right|<\infty.
\end{equation*}
These integrability conditions are the same as in \cite[Lemma
2.1]{papadatos1995distance}; see also  \cite{CP89}.

  The  Stein kernel \eqref{eq:15} has a number of remarkable
  properties. In particular,  it plays a pivotal role in the connection between
  information inequalities and Stein's method, see
  \cite{ledoux2014stein,nourdin2013integration,nourdin2013entropy}. 

  \begin{proposition}\label{prop:stek}
    Let {Assumptions 1-5 hold}. 
    Suppose furthermore
    that there exists  $\delta>0$ such that  $\mathcal{D}^{\star}(a\, Id+b) = a\, \delta$ for all $a, b
    \in \R$ and $Id(x) = x$ the
    identity. Then 
    \begin{equation}
      \label{eq:45}
      \E \left[ \tau(X) \mathcal{D}^{\star}g(X-l) = \E \left[ (X- \nu)
        g(X)\right] \right]
    \end{equation}
for all $g \in dom(\mathcal{D}, X, \tau)$ and
  \begin{equation}\label{eq:46}
    \E \left[ \tau(X) \right] = \delta^{-1} \mbox{Var}(X).
  \end{equation}
  \end{proposition}
  \begin{proof}
    Identity \eqref{eq:45} is obvious and \eqref{eq:46} follows by
    taking $g(x-l) =  x-\nu$ (which is allowed) in \eqref{eq:45}.
  \end{proof}
  \begin{remark}\label{rem:positstek}
It is easy to show that, moreover,  $\tau(x)\ge0$ if $\mathcal{D}$ is
either the strong derivative or the discrete forward/backward
difference. 
  \end{remark}
\begin{remark}\label{rem:steke}
  Although the function $\tau = \mathcal{T}_X^{-1}(\nu-Id)$ has
  already been much used in the literature, it has been given various
  names all marked with some ambiguity.  Indeed
  \cite{NoPe09,NP11,nourdin2013entropy} (among others) refer to $\tau$
  as the ``Stein factor'' 
  despite the fact that this 
  {term} also refers to the bounds on
  the solutions of the Stein equations, see
  \cite{RO07,daly2008upper,BGX13}. Other authors, including
  \cite{CP95,cacoullos1998short,CPP01}, rather refer to this function
  as the ``$\omega$-function'' or the ``covariance kernel'' of $X$. We
  prefer to unify the terminology by calling $\tau$ a \emph{Stein
    kernel}.
\end{remark}
Two particular {instances} of \eqref{eq:tau}  have already been perused in
the literature in the following case.  
\begin{definition}[Pearson's class of  distributions]\label{def:pearson-class}
  A continuous distribution $p$ with   support $supp(p)$ is a member of
  Pearson's family of distributions if it is solution to the
  differential equation 
  \begin{equation}
    \label{eq:27}
    \frac{p'(x)}{p(x)} =
    \frac{\alpha-x}{\beta_2(x-\lambda)^2+\beta_1(x-\lambda) + \beta_0}
  \end{equation}
for some constants $\lambda, \alpha, \beta_j, \, j=0, 1, 2$. 
\end{definition}
Properties of the differential operator $\mathcal{T}_Xf = (fp)'/p$
have 
{been} studied in quite some detail for distributions $p$
which belong to Pearson's class of distributions, see
e.g. \cite{DZ91,korwar1991characterizations,johnson1993note,papathanasiou1995characterization,lefevre2002generalized,loh2004characteristic,APP11}.
If $X \sim p$, a Pearson distribution, then by definition its
derivative $p'$ exists and, using 
 \eqref{eq:11}, its canonical Stein operator is
\begin{equation*}
  \mathcal{T}_Xf(x)  = f'(x) +
  \frac{\alpha-x}{\beta_2(x-\lambda)^2+\beta_1(x-\lambda) + \beta_0} f(x)
\end{equation*}
for $x \in \mbox{supp}(p)$. In general this operator is not easy to
handle. 

It is shown in \cite{korwar1991characterizations} that, in the setting
of Example \ref{ex:cpont},  a density $p$ satisfies 
  \eqref{eq:27} if and only if its Stein
kernel $\tau(x)$ is quadratic. This function can be calculated (using
e.g. \cite[equation (3.5)]{DZ91}), and is given by
  \begin{equation*}
    \tau(x)  = \frac{\beta_0+\beta_1x + \beta_2x^2}{1-2 \beta_2},
  \end{equation*}
see  also \cite{papathanasiou1995characterization}.
This observation leads us to considering distributions, discrete or
continuous, which have Stein kernel of the form 
  \begin{equation}\label{eq:31}
\mathcal{T}_X^{-1}(\nu-Id)(x) = a+ bx + cx^2
  \end{equation}
  for some constants $a, b$ and $c$.  
For distributions
satisfying \eqref{eq:31} we deduce  a natural family of Stein operators
\begin{equation*}
  \mathcal{A}_Xg(x) = \left( a+ bx + cx^2 \right) \mathcal{D}^{\star} g(x) + (\nu-x) g(x)
\end{equation*}
acting on the class $\mathcal{F}(\mathcal{A}_X)$ of 
functions such that $g \tau \in \mathcal{F}(X)$ as well as 
\begin{equation*}
  \E \left|g(X) (\nu-X)  \right|<\infty \mbox{ and } \E \left| \mathcal{D}^{\star}g(X)
   \left( a+ bX + cX^2 \right)\right|<\infty. 
\end{equation*}

  \begin{remark}
    \cite{S01,APP11} call the class of densities satisfying
    \eqref{eq:31} the Pearson class (in the continuous case) and the
    Ord class (in the discrete case); Ord's class as
    originally defined in \cite{ord1967system} is, in fact, larger. In
    the case of integer valued random variables, \cite[Theorem
    4.6]{korwar1991characterizations} shows that, under conditions on
    the coefficients, condition \eqref{eq:31} is equivalent to
    requiring that $p(x) = \binom{a}{x}\binom{b}{n-x}/\binom{a+b}{n}$
    for some constants $a, b$ and $n$, {so that} $X$ has a
    generalized hypergeometric distribution.  See also
    \cite{AfBalPa2014} where  distributions satisfying  \eqref{eq:31}  are referred to as 
    Cumulative Ord distributions; see in particular their Proposition
    2.1 for a characterization.
  \end{remark}
\begin{example}\label{ex:listoftaus}
  Many ``useful'' densities satisfy \eqref{eq:31} in which case
  operator \eqref{eq:tau} has a nice form as well. The 
following examples are easy to compute and will be useful in the 
sequel; for future reference we also provide the log derivative of
the density. 
\begin{enumerate}
\item Continuous setting, strong derivative : 

\begin{enumerate}
\item  Gaussian $\mathcal{N}(0, \sigma^2)$ with $p(x) = (2\pi)^{-1}
  e^{-x^2/2}$ on $\mathcal{I} = \R$~:  
  \begin{equation*}
    \frac{p'(x)}{p(x)}= -\frac{x}{\sigma^2}
\mbox{ and } \tau(x) =  \sigma^2;
  \end{equation*}
\item  Gamma $\Gamma(\alpha, \beta)$  with $p(x) = \beta^{-\alpha} \Gamma(\alpha)^{-1}
  e^{-x/\beta}x^{\alpha-1}$ on $\mathcal{I} = \R^+$~: 
  \begin{equation*}
 \frac{p'(x)}{p(x)} = \frac{-1+\alpha}{x} -
  \frac{1}{\beta} \mbox{ and }  \tau(x) = \frac{x}{\beta};
  \end{equation*}
\item  Beta $\mathcal{B}(\alpha, \beta)$ with $p(x) = B(\alpha,
  \beta)^{-1}x^{\alpha-1}(1-x)^{\beta-1}$  on $\mathcal{I} = [0,1]$~:  
  \begin{equation*}
 \frac{p'(x)}{p(x)} = \frac{\alpha-1}{x}
  - \frac{\beta-1}{x-1}
 \mbox{ and }\tau(x) =\frac{ x(1-x)}{\alpha+\beta};
  \end{equation*}
\item  Student  $t_{t}$ (for $t>1$) with $p(x) = \nu^{-1/2} B(\nu/2,
  1/2)^{-1}(\nu/(\nu+x^2))^{(1+\nu)/2}$ on $\R$~:  
  \begin{equation*}
 \frac{p'(x)}{p(x)} = - \frac{x(1+ t)}{t + x^2} \mbox{ and }
 \tau(x) = \frac{x^2+t}{t-1}. 
  \end{equation*}
\end{enumerate}
\item Discrete setting, forward derivative : 
  \begin{enumerate}
\item   Poisson $Po(\lambda)$ with $p(x) = e^{-\lambda} \lambda^x/x!$
  on $\mathcal{I} = \Z$~: 
  \begin{equation*}
    \frac{\Delta^+p(x)}{p(x)} = \frac{\lambda}{x+1}-1 \mbox{ and
    }\tau(x) =x;
  \end{equation*}
\item  Binomial $Bin(n, p)$ with $p(x) = \binom{n}{x}p^x(1-p)^{n-x}$
  on $\mathcal{I} = [0, n] \cap \Z$~: 
  \begin{equation*}
   \frac{\Delta^+p(x)}{p(x)} = \frac{(n-x)}{x+1}\frac{p}{1-p}-1 \mbox{ and
    } \tau(x) = (1-p)x.
  \end{equation*}
  \end{enumerate}
  \end{enumerate}

\end{example}

\subsection{Invariant measures of diffusions}\label{rem:edvikutu}
Recent papers \cite{EdVi12,KuTu11,kusuoka2013extension} provide a
general framework for performing Stein's method with respect to
densities $p$ which are supposed to admit a variance and be
continuous (with respect to the Lebesgue measure), 
 bounded with open interval support. Specifically, \cite{KuTu11}
suggest studying operators of the form
\begin{equation}\label{eq:39}
    \mathcal{A}_Xg(x) = \frac{1}{2}\beta(x)g'(x) + \gamma(x) g(x) 
  \end{equation}
  with $\gamma\in L^1(\mu)$ a continuous function with strictly one
  sign change on the support of $X$, negative on the right-most
  interval and such that $\gamma p$ is bounded and $\E [\gamma(X)] =
  0$,
  \begin{equation*}
    \beta(x) =  \frac{2}{p(x)}\int_a^x\gamma(y)p(y)dy,
  \end{equation*}
for $g \in \mathcal{F}(\mathcal{A}_X)$ the class of functions such
that $g \in C^1$ and 
\begin{equation*}
  \E |\gamma(X) g(X)|<+\infty \mbox{ and } \E
|\beta(X)g'(X)|<+\infty.
\end{equation*}
Then \cite{KuTu11} (see as well \cite{kusuoka2013extension} for an
extension) use 
diffusion {theory} to prove that
$\mathcal{A}_X$ are indeed Stein operators in the sense of the
Introduction (their approach falls within the generator approach).  In
our framework, \eqref{eq:39} is a particular case of \eqref{eq:22},
with ${f}= \beta/2 = \mathcal{T}_X^{-1}\gamma \in \mathcal{F}(X)$ and
$\gamma= \mathcal{T}_X{f}$ (which necessarily satisfies
$\E[\gamma(X)]=0$) and $\mathcal{F}(\mathcal{A}_X) = dom ((\cdot)', X,
{f})$.

\subsection{Gibbs measures on non-negative integers}\label{sec:gibbs-measures-non}
We can treat any discrete univariate distribution on non-negative
integers by writing it as a Gibbs measure
$$ \mu(x) = \frac{1}{\kappa} \exp(V(x)) \frac{\omega^x}{x!}, \quad x=0, 1, \ldots, N,$$
where $N \in \{0, 1, 2, \ldots\} \cup \{\infty\}$ and $\kappa$ is a
normalizing constant. Here the choice of $V$ and $\omega$ is not
unique. In \cite{eichelsbacher2008stein}, Stein's method for discrete
univariate Gibbs measures on non-negative integers is developed, with
operator
\begin{equation}
  \label{eq:41}
  \mathcal{A}_{\mu}(f)(x) =  f(x+1)\omega \exp{(V(x+1) - V(x))}  - x f(x)
\end{equation}
acting on the class of functions such that $f(0)=0$ and, in case $N$ is infinite,
$\lim_{x\to\infty}f(x)\exp(V(x)) \frac{\omega^x}{x!}=0$.
 The canonical
operator \eqref{eq:3} is (with $\cD = \Delta^+$) 
$${\mathcal{T}}_\mu f(x) =  f(x+1)\frac{\omega}{x+1} \exp{(V(x+1) - V(x))}  -  f(x)$$
which yields \eqref{eq:41} via \eqref{eq:9} using the pair  $(f(x),
g(x)) = (f(x), x+1)$. 
In \cite{eichelsbacher2008stein}, other choices of birth and death rates were
discussed; here the birth rate $b_x$ is the pre-factor of $g(x+1)$,
and the death rate $d_x$ is the pre-factor of $g(x)$.  Indeed any
choice of birth and death rates  which satisfy the detailed
balance conditions
$$\mu(x) b_x = \mu(x+1) d_{x+1}$$
for all $x$ are viable. Our canonical Stein operator can be written as 
$${\mathcal{T}}_\mu {g}(x) = \frac{b_x}{d_{x+1} }{g}(x+1)  -  {g}(x).$$
Choosing ${f}(x) =d_x$ and applying \eqref{eq:22} gives the general
Stein operator $ b_x {g}(x+1) - d_x {g}(x).$ The Stein kernel here is
$$ \tau(x) = \sum_{y=0}^x e^{V(y)-V(x)} \frac{x!}{y! w^{x-y}} (\nu - y)$$ 
with $\nu$ the mean of the distribution. This expression can be
simplified in special cases; for example in the Poisson case $V$ is
constant and we obtain $\tau(x) = w$, as before.
Similar developments are also considered by \cite{Ho04}.

\subsection{Higher order operators}
\label{sec:high-order-oper}
So far, in  all examples provided we only consider first-order
difference or differential operators. One way
of constructing higher order operators is to consider 
\begin{equation*}
  \mathcal{A}_Xf = \mathcal{T}_X(c \mathcal{D}^kf)
\end{equation*}
for $c$ well chosen and $\mathcal{D}^k$ the $k$th iteration of
$\mathcal{D}$. This approach is strongly connected with
Sturm-Liouville theory and will be the subject of a future
publication. Here we merely give  examples illustrating that our
results are not restricted to first-order operators. {The first example is the Kummer-$U$ distribution in Example \ref{ex:kummerv1}.}

Similar considerations as in Example 
\ref{ex:kummerv1} provide tractable operators for other distributions involving special
functions. 
\begin{example}[Variance Gamma distribution] Let $K_\nu$
be the modified Bessel function of the second kind, of index $\nu$. A
random variable has the  variance gamma distribution  $VG(\nu, \alpha,
\beta, \eta)$ if its density is given on $\R$ by
$$ p(x) = \frac{(\alpha^2 - \beta^2)^{\nu + \frac12}}{\sqrt{\pi}\Gamma\left(\nu + \frac12\right)} \left( \frac{|x - \eta|}{2 \alpha}\right)^\nu e^{\beta x} K_\nu(\alpha |x - \eta|),$$
where $\alpha >|\beta| > 0, \nu > - \frac12, \eta \in \R$. 
For simplicity we take $\eta = 0, \alpha =1, $ and $\nu > 0$. 
A generator for this distribution is 
\begin{equation}\label{eq:gauntss}
  \AAA f(x) = x f''(x) + (2 \nu + 1 + 2 \beta x) f'(x) + \{ (2 \nu + 1) \beta - (1 - \beta^2) x \} f(x),
\end{equation}
see  \cite{gaunt2014variance}. The canonical operator is (with
$\mathcal{D}$ the usual strong derivative)
\begin{equation*}
  \mathcal{T} (f) (x) = f'(x) + f(x) \left(  \frac{2 \nu}{x} + \beta
  \right) - \frac{K_{\nu+1}(x)}{K_\nu(x)}. 
\end{equation*}
Applying \eqref{eq:9} via the pair $(f, g) = (f, g(f))$ with 
with 
\begin{equation*}
g(f)(x) = x \frac{f'(x)}{f(x)} + x \left(  \beta  +
  \frac{K_{\nu+1}(x)}{K_\nu(x)} \right)
\end{equation*}
we retrieve \eqref{eq:gauntss}. 
\end{example}
\subsection{Densities satisfying a differential
  equation} \label{sec:distr-satisfy-diff}

Lastly we consider the case where the density of interest $p$ with
interval support $I = \left\{ a, b \right\}$ is defined as the
solution of some differential equation, say
\begin{equation*}
\mathcal{L}(p) = 0
\end{equation*}
along with some boundary conditions. Suppose that
$\mathcal{L}$ admits an adjoint (w.r.t. Lebesgue integration)
which we denote $\mathcal{L}^{\star}$ so that, for   $X \sim p$, we
can   apply  integration by parts to get 
\begin{align*} 
 0 &=   \int_a^b g(x) \mathcal{L}(p)(x) dx  = C_a^b(g,p) + \int_a^b
 \mathcal{L}^{\star}(g)(x) p(x) dx \\
& = C_a^b(g,p) + \E \left[  \mathcal{L}^{\star}(g)(X)\right]
\end{align*}
with $C_a^b(g, p)$ the constant arising through the integration by
parts. We define 
$
  \mathcal{A}_X(g) = \mathcal{L}^{\star}(g)
$
acting on the class $\mathcal{F}(\mathcal{A}_X)$ of sufficiently
smooth functions such
that $C_a^b(g, p)=0$. To qualify $\mathcal{A}_X$  as a
Stein operator in the sense of \eqref{steingen},  it still remains to
identify conditions on $g$ which ensure that this operator
characterises the density.

This point of view blends smoothly into our canonical approach to
Stein operators; we can moreover provide conditions on $g$ in a
generic way. To see this {fix a function $g$ of interest and} choose $f$ such that
\begin{equation*}
  \frac{(fp)'}{p} =  \mathcal{L}^{\star}(g)
\end{equation*}
{if such an $f$ exists}. Then, reversing the integration by parts argument
provided above, we get 
\begin{align*}
  f(x) & = \frac{1}{p(x)}\int_a^x  \mathcal{L}^{\star}(g)(u) p(u)du \\
  & = \frac{1}{p(x)}C_a^x(g, p) + \frac{1}{p(x)}\int_{a}^x g(u) \mathcal{L}(p)(u) du \\
  &= \frac{1}{p(x)}C_a^x(g, p) =:F(g, p)(x)
\end{align*}
with $\frac{1}{p(x)}C_a^x(g, p)$ the quantities resulting from the integration
by parts (and using the fact that now $\mathcal{L}(p)=0$, by
assumption). This leads to the standardization  
\begin{equation*}
  \mathcal{A}_X(g) = \mathcal{T}_X \left( F(g, p)  \right)
\end{equation*}
acting on the class of functions 
$
  \mathcal{F}(\mathcal{A}_X) = \left\{
  g \mbox{ such that } F(g, p) \in \mathcal{F}(X)
\right\}.
$
Note how, in particular, the assumption $F(g, p) \in \mathcal{F}(X)$
implies that $C_a^b(g, p)=0$, as demanded in the beginning of the
Section.

\begin{example}
We illustrate this point of view in the case of
the spectral density  $h_n$ on $[-2, 2]$ of a $GUE(n, 1/n)$ random matrix  studied
in \cite{GoTi06,haagerup2012asymptotic}.  This density is defined through the
third order differential equation 
\begin{equation*}
  \mathcal{L}(h_n)(x) =   \frac{1}{n^2} h_n'''(x) + (4-x^2)
  h_n'(x) + x h_n(x) = 0, x \in \R, 
\end{equation*}
along with a boundary condition.
Letting $X \sim h_n$  it is straightforward to show that 
\begin{equation*}
\mathcal{L}^{\star}(g)(x) = -\frac{1}{n^2}g'''(x) -((4-x^2)g(x))' + x g(x)
\end{equation*}
acting on the collection 
\begin{equation*}
\left\{ g \in C^3 \mbox{ such that }
    \left[ 
\frac{h_n''(x)g(x)- h_n'(x)g'(x) }{n^2}+h_n(x) g(x)
(4-x^2) 
\right]_{-2}^2 = 0\right\}.  
\end{equation*}
Integrating by parts we then get 
\begin{equation*}
 F(g, h_n)(x)   =  \frac{1}{n^2} \left( g''(x)-g'(x)
     \frac{h_n'(x)}{h_n(x)}+      \frac{h_n''(x)}{h_n(x)}g(x) 
 \right) + (4-x^2)g(x)  - \mathrm{c}
\end{equation*}
with $\mathrm{c}=g''(-2)
     +g'(-2) \frac{h_n'(-2)}{h_n(-2)} -
     \frac{h_n''(-2)}{h_n(-2)}g(-2)$. 
Considering only functions $g$ such that $ F(g, h_n) \in
\mathcal{F}(X)$   leads to a Stein operator for $X$.   
\end{example}



\section{{Distributional comparisons}}
\label{sec:grdistr-comp}

Resulting from our framework, in this Section we provide a general
``comparison of generators approach'' (Theorem \ref{theo:general-1})
which provides bounds on the probability distance between univariate
distributions in terms of their Stein operators.  This result is
formal and abstract; it is our take on a general version of Part B of
Stein's method.  Specific applications to concrete distributions will
be deferred to Section~\ref{sec:examples-again}.

 \subsection{Comparing Stein operators} \label{sec:comparing-operators}

 Let $(\cX_1, \mathcal{B}_1, \mu_1)$ and $(\cX_2, \mathcal{B}_2,
 \mu_2)$ be two measure spaces as in Section~\ref{sec:setup}.  Let
 $X_1$ and $X_2$ be two random variables on $\cX_1$ and $\cX_2$,
 respectively, and suppose that their respective densities $p_1$ and
 $p_2$ have interval support.
Let $\mathcal{D}_1$ and $\mathcal{D}_2$ be two linear operators acting
on $\cX_1$ and $\cX_2$ and satisfying Assumption~\ref{ass:prodrule}
(with $l_1$ and $l_2$, respectively) and Assumption~\ref{ass:ker}.  Denote by
$\mathcal{T}_1$ and $\mathcal{T}_2$ the Stein operators associated
with $(X_1, \cD_1)$ and $(X_2, \cD_2)$, acting on Stein classes
$\mathcal{F}_1=\mathcal{F}(X_1)$ and $\mathcal{F}_2=\mathcal{F}(X_2)$,
respectively.  Finally let $\E_ih = \E h(X_i)$ denote the expectation
of a function $h$ under the measure $p_i d\mu$, $i=1,2$.

The framework outlined in Section~\ref{sec:gener-vers-steins}
(specifically Section \ref{sec:stein-oper-stein}) is
tailored for the following result to hold. 
\begin{theorem}\label{theo:general-1}
Let $h$ be a function such that
$\E_i|h|<\infty$ for $i=1,   2$. 
  \begin{enumerate}
  \item Let $(f,g) $ with $f \in \mathcal{F}_1$ and $g\in dom
    (\mathcal{D}_1, X_1, f)$ solve the $X_1$-Stein equation
    \eqref{eq:genstequationfg} for $h$. Then
\begin{align}
  \label{eq:realstart}
\E_2h - \E_1h &  =  \E_2 \left[
  f (X_2)\cD_1^{\star}g(X_2)- g(X_2)\mathcal{T}_1f  (X_2)  
\right] .
\end{align}
  \item Fix $f_1 \in \mathcal{F}_1$ and define the function  $
    g_h:=\frac{1}{f_1}
  \mathcal{T}_1^{-1} (h - \E_1 h)$. Then 
\begin{align}
  \label{eq:start}
\E_2h - \E_1h &  = \E_2 \left[
  f_1(\cdot)\cD_1^{\star}g_h(\cdot)-f_2(\cdot)\cD_2^{\star}g_h(\cdot)
\right. \\ 
& \quad \quad \quad + \left. 
   g_h(\cdot)\mathcal{T}_1f_1(\cdot)-    g_h(\cdot)\mathcal{T}_2f_2(\cdot)  \right]\nonumber
\end{align}
for all $f_2\in \mathcal{F}_2$ such that $g_h \in dom(\cD_2,X_2,f_2)$.

\item  Fix $g_1 \in dom(\mathcal{D}_1, X_1)$ and define the function
  $ f_h:=\frac{1}{g_1}
  \mathcal{T}_1^{-1} (h - \E_1 h)$.  If $f_h \in   
 \mathcal{F}_1\cap\mathcal{F}_2$
then 
\begin{align}
  \label{eq:start2}
\E_2h - \E_1h &  = \E_2 \left[
  f_h(\cdot)\cD_1^{\star}g_1(\cdot)-f_h(\cdot)\cD_2^{\star}g_2(\cdot)
\right. \\ 
& \quad \quad \quad + \left. 
   g_1(\cdot)\mathcal{T}_1f_h(\cdot)-    g_2(\cdot)\mathcal{T}_2f_h(\cdot)  \right].\nonumber
\end{align}
for all $g_2 \in
dom(\mathcal{D}_2, X_2)$.
  \end{enumerate}
\end{theorem}
\begin{remark}
  \label{rem:takingf0} Our approach contains the classical ``direct''
  approach described in the Introduction (see \eqref{eq:firstepte}).
  Indeed, if allowed, one can take $f_1 = 1$ and $f_2=0$ in
  \eqref{eq:start} to get
  \begin{equation*}
    \E_2h - \E_1h = \E_2 \left[ \mathcal{D}_1^{\star}g_h(\cdot) +
      u_1(\cdot) g_h(\cdot) \right]  \end{equation*}
with $u_1$ the score of $X_1$ (defined in  \eqref{eq:10}) and $g_h$
now the usual solution of the Stein equation. This 
yields the bound 
\begin{equation*}
  d_{\mathcal{H}}(X_1, X_2) \le \sup_{\mathcal{H}} \left| \E_2 \left[ \mathcal{A}(g_h)(X_2) \right]\right|
\end{equation*}
with $\mathcal{A}(g_h) = \mathcal{D}_1^{\star}g_h + u_1
g_h$. In this case one does not
need to calculate $\mathcal{T}_2$.
\end{remark}
\begin{proof}
  The starting point is the Stein equation \eqref{eq:genstequationfg}
  which, in the current context, becomes
  \begin{equation}\label{eq:18}
    h(x) - \E h(X_\bullet) = f(x) \mathcal{D}_\bullet^{\star}g(x) + g(x)
    \mathcal{T}_\bullet f(x) = \frac{\mathcal{D}_\bullet \left( f
        g p_\bullet) \right)}{p_\bullet}(x) 
  \end{equation}
  with $\bullet \in \left\{ 1, 2 \right\}$.  Solutions of this
  equation are pairs of functions $(f, g)$ with $f \in
  \mathcal{F}(X_\bullet)$ and $g\in dom (\mathcal{D}_\bullet,
  X_\bullet, f)$. Using $\bullet=1$, replacing $x$ by $X_2$ and taking
  expectations gives \eqref{eq:realstart}.

For \eqref{eq:start}, first fix $f_1 \in \mathcal{F}_1$ and choose $g = g_h$ the
corresponding solution of \eqref{eq:18} with $\bullet = 1$.  By
construction we can then take
expectations and write
\begin{equation*}
  \E h(X_2)  - \E h(X_1) = \E \left[ f_1(X_2)
    \mathcal{D}_1^{\star}g_h(X_2) + g_h(X_2) 
    \mathcal{T}_1 f_1(X_2) \right]
\end{equation*}
because $h \in L^1(X_1) \cap L^1(X_2)$.  
Finally we know that for  all $f_2 \in \mathcal{F}_2$ such that $g_h \in
dom(\mathcal{D}_2, X_2, f_2)$ we
can  use \eqref{eq:18}
with $\bullet = 2$ to get 
\begin{align*}
  \E\left[ f_2(X_2) \mathcal{D}_2^{\star}g_h(X_2) + g_h(X_2)
    \mathcal{T}_2 f_2(X_2) \right]  = 0. 
\end{align*}
Taking differences we get \eqref{eq:start}. Equation  \eqref{eq:start2}
follows in a similar fashion, fixing this time $f = f_h$ and letting
$g_1$ and $g_2$ vary. 
\end{proof}

The power of Theorem \ref{theo:general-1} and of Stein's method in
general lies in the freedom of choice on the r.h.s. of the identities
: all functions $f_\bullet, g_{\bullet}$ (where now $\bullet$ needs to
be replaced by $h, 1$ or $2$ according to which of \eqref{eq:start} or
\eqref{eq:start2} is used) can be chosen so as to optimise resulting
bounds. We can even optimise the bounds over all suitable pairs
$(f,g)$.  We will discuss two particular choices of functions in
Section \ref{sec:comp-stein-kern} which lead to well-known Stein
bounds. We also will provide illustrations (discrete
vs discrete, continuous vs continuous and discrete vs continuous) in
Section \ref{sec:examples-again}.

In particular \eqref{eq:start} and \eqref{eq:start2} provide tractable
(and still very general) versions of \eqref{eq:firstepte2}. Indeed
taking suprema over all $h \in \mathcal{H}$ some suitably chosen class
of functions we get, in the notations of the Introduction,
\begin{equation*}
  d_{\mathcal{H}}(X_1, X_2)=\sup_{h\in\mathcal{H}}|\E_2 h-\E_1h| \le
A_1 + A_2
\end{equation*}
with $$A_1 = A_1(\mathcal{H}) = \sup_{h \in \mathcal{H}}\left| \E_2
  \left[
    f_\bullet(\cdot)\cD_1^{\star}g_\bullet(\cdot)-f_\bullet(\cdot)\cD_2^{\star}g_\bullet(\cdot)
  \right] \right| $$ and $$A_2 = A_2(\mathcal{H}) = \sup_{h \in
  \mathcal{H}} \left| \E_2 \left[
    g_\bullet(\cdot)\mathcal{T}_1f_\bullet(\cdot)-
    g_\bullet(\cdot)\mathcal{T}_2f_\bullet (\cdot)\right] \right|
.$$  Different choices of functions $f_1$ and $f_2$ (resp. $g_1$ and
$g_2$) will lead to different expressions bounding all distances
$d_{\mathcal{H}}(X_1, X_2)$ in terms of properties of $\mathcal{T}_1$
and {$\mathcal{T}_2$}. 
\begin{remark}
  If there exist no functions $f_1, f_2$ (resp. $g_1$, $g_2$) such
  that the assumptions are satisfied, then the claims of
  Theorem~\ref{theo:general-1} are void. Such is not the case whenever {$p_1$ and $p_2$}
  are ``reasonable''.
\end{remark}


\begin{remark}[About the Stein factors]
  In view of \eqref{eq:start} and \eqref{eq:start2}, good
bounds on $\E_1h - \E_2h$ will depend on the control we have on
functions 
\begin{equation}
  \label{eq:gensol}
  g_h = \frac{\mathcal{T}_1^{-1} \left( h-\E_1 h \right)}{f_1} \mbox{
    and/or  }   f_h = \frac{\mathcal{T}_1^{-1} \left( h-\E_1 h \right)}{g_1}.
\end{equation}
Bounds on these functions and on their derivatives are called
dedicated literature,
\emph{Stein (magic) factors} (see for example
\cite{daly2008upper,RO12}).  There is an important connection between
such constants and Poincar\'e / variance bounds / spectral gaps, as
already noted for example in
\cite{Ch80,K85,MR2128239,lefevre2002generalized,MR2527030}. This
connection is quite transparent in our framework and will be explored
in future publications.
\end{remark}

In the sequel we will not use the full freedom of choice provided by
Theorem~\ref{theo:general-1}, but rather focus on applications of
identity \eqref{eq:start} only. Indeed in this case much is known
about $\|g_h\|$ and $\| \mathcal{D}g_h\|$ in case $f_1=1$ and $X_1$ is
Gaussian (see \cite{ChGoSh11}), Binomial (see \cite{ehm1991binomial}),
Poisson (see \cite{BaHoJa92}), Gamma (see \cite{ChFuRo11,NP11}), etc.
See also \cite{EdVi12,KuTu11,Do14,LS12a} for computations under quite
general assumptions on the density of $X_1$.  We will make use of
these results in Section \ref{sec:examples-again}.  It is hopeless {to} 
wish for useful bounds on \eqref{eq:gensol} in all generality (see
also the discussion in \cite{arras2016stein}).  Of course one could
proceed as in \cite{Do14,CS11} or \cite{LS12a} by imposing \emph{ad
  hoc} assumptions on the target density which ensure that the
functions in \eqref{eq:gensol} have good properties.  Such approaches
are not pursued in this paper.  Specific bounds will therefore only be
discussed in particular examples.

\subsection{Comparing Stein kernels and score functions}\label{sec:comp-stein-kern}
There are two obvious ways to exploit \eqref{eq:start}, namely either
by trying to make the first summand equal zero, or by trying to make
the second summand equal zero. In the rest of this section we do just
that, in the case $\mathcal{X}_1 = \mathcal{X}_2$ and $\mathcal{D}_1 =
\mathcal{D}_2= \mathcal{D}$ (and hence $l_1=l_2=l$); extension of this
result to mixtures is straightforward.
 

Cancelling the first term in \eqref{eq:start} and ensuring that all
resulting assumptions are satisfied immediately leads to the following
result.

\begin{corollary}\label{cor:compscor}
  Let $\mathcal{H}\subset L^1(X_1) \cap L^1(X_2)$.   Take $f \in
  \mathcal{F}_1\cap \mathcal{F}_2$ and suppose that
  $(1/f)\mathcal{T}_1^{-1}(h-\E_1h) \in dom(\mathcal{D}, X_1, f) \cap
  dom(\mathcal{D}, X_2, f)$   for all $h \in \mathcal{H}$. Then
  \begin{equation}
    \label{eq:boundT1T2}
     \sup_{h\in \mathcal{H}} | \E_1 h - \E_2 h | \le
     \kappa_{\mathcal{H}, 1}(f)   \E_2 | \mathcal{T}_1f - \mathcal{T}_2f|
  \end{equation}
with $\kappa_{\mathcal{H}, 1}(f) = \sup_{h \in \mathcal{H}} \| (1/f)\,\mathcal{T}_1^{-1}(h-\E_1h) \|_{\infty}$.
\end{corollary}

\begin{remark}
\begin{enumerate}
\item 
If the constant function
$1 \in \FFF_1 \cap \FFF_2$, then we can take  
$f=1$ in \eqref{eq:boundT1T2} to deduce that 
\begin{equation*}
  d_{\mathcal{H},1}(X_1, X_2) \le \kappa_{\mathcal{H}, 1}(1) \E_{2} \left|
    u_1 -  u_2 \right|\leq \kappa_{\mathcal{H}, 1} \sqrt{\E_{2} \left[ \left(
    u_1 -  u_2 \right)^2 \right]},
\end{equation*}
with $u_i = \mathcal{T}_i(1)$ the score function of $X_i$ (defined in
\eqref{eq:10}) and $\kappa_{\mathcal{H},1}$ an explicit constant that
can be computed in several important cases, see \mbox{e.g.}
\cite[Section 4]{LS12a} and  \cite{Sh75,MR2128239} for
applications in the Gaussian case. Note that $\mathcal{J}(X_1, X_2) =
\E_{2} \left[ (u_1 - u_2)^2 \right]$ is the so-called generalized
Fisher information distance (see e.g. \cite{Jo04,LS12a}).
\item
The assumption that $f \in
  \mathcal{F}_1\cap \mathcal{F}_2$ can be relaxed; if $\int_I D_2(fp_2) d\mu \ne 0$ then this just adds terms which relate to the boundaries of $I$. 
\end{enumerate}
\end{remark} 



Cancelling the second term in \eqref{eq:start} and ensuring that all
resulting assumptions are satisfied immediately leads to the following
result.

\begin{corollary}\label{cor:compstfac}
  Let $\mathcal{H}\subset L^1(X_1)\cap L^1(X_2)$.   Take $\omega \in
   Im(\mathcal{T}_1)\cap Im(\mathcal{T}_2)$ such that
   $\mathcal{T}_1^{-1}(h-\E_1h)/ \cT_1^{-1}(\omega)\in
   dom(\cD,X_1,\cT_1^{-1}(\omega)) \cap    dom(\cD,X_2,\cT_2^{-1}(\omega))$. Then 
  \begin{equation}\label{eq:stefact} 
    | \E_1 h - \E_2 h | \le \kappa_{\mathcal{H}, 2}(\omega)
     \E_2 | \cT_1^{-1}(\omega)-\cT_2^{-1}(\omega)|
  \end{equation}
with $\kappa_{\mathcal{H}, 2}(\omega) = \sup_{h \in \mathcal{H}} \|
 \cD\left(\mathcal{T}_1^{-1}(h-\E_1h) / \cT_1^{-1}(\omega)  \right)\|_{\infty}$.
\end{corollary} 
 
If, moreover, $X_1$ and $X_2$ have common finite mean $\nu$ then one can choose $\omega(x) =
\nu-x$ 
 in \eqref{eq:stefact} to get 
\begin{equation}
  \label{eq:6}
   | \E_1 h - \E_2 h | \le \kappa_{\mathcal{H}, 2}
     \E_2 |\tau_1-\tau_2|
\end{equation}
with $\tau_j$, $j=1, 2$, the Stein kernel of $X_j$ (defined in
\eqref{eq:15}) and $\kappa_{\mathcal{H}, 2}$ an explicit constant that
can be computed in several cases. In \cite{CPP01}, and references
therein cited, consequences of \eqref{eq:6} are explored in quite some
detail. In particular in the Gaussian and central Gamma cases,
\eqref{eq:6} has been exploited fruitfully in conjunction with
Malliavin calculus, leading to an important new stream of research
known as ``Nourdin-Peccati analysis'',
see \cite{NP11,NoPe09}. See also aforementioned references
\cite{kusuoka2013extension,KuTu11,EdVi12} where several extensions of
the Nourdin-Peccati analysis are discussed. Note that, in the Gaussian
case $X_1\sim \mathcal{N}(0, 1)$ we readily obtain $\tau_1 = 1$. The
quantity 
\begin{equation}
  \label{eq:38}
  S(X) = \sqrt{ \E \left[\left( 1-\tau_2 \right)^2\right]}
\end{equation}
is the   \emph{Stein discrepancy} from
\cite{nourdin2013entropy,ledoux2014stein}.

\subsection{Sums of {independent} random variables and 
  the Stein kernel}\label{sec:sum}

We begin by  relaxing the definition of Stein kernel. This approach is
similar to that advocated in \cite{nourdin2013integration}. 
\begin{definition}
  Let $\mathcal{X}$ be a set and $\mathcal{D}$ a linear operator
  acting on $\mathcal{X}^{\star}$ satisfying the Assumptions of
  Section \ref{sec:setup}. Let $X\sim p$ have mean $\nu$ and
  $\mathcal{D}$-Stein pair $(\mathcal{T}_X, \mathcal{F}(X))$. A random
  variable $\tau_X(X)$ is a $\mathcal{D}$-Stein kernel for $X$ if it
  is measurable in $X$ and if
  \begin{equation}\label{eq:51}
\E \left[ \tau(X) \mathcal{D}^{\star}g(X-l) = \E \left[ (X- \nu)
        g(X)\right] \right]
    \end{equation}
    for all $g \in dom(\mathcal{D}, X, \tau)$.  If, moreover,
    $dom(\mathcal{D}, X, \tau)$ is dense in $L^1(\mu)$ then the
    Stein kernel is unique.
\end{definition}
  Applying \eqref{eq:8} one immediately sees that $\mathcal{T}_p^{-1}(Id
- \nu)$ is a Stein kernel for $X$. 
\begin{proposition} If $\mathcal{D}^{\star}$ satisfies a chain rule
  $\mathcal{D}^{\star}f(ax) = a \mathcal{D}^{\star}_af(x)$
for some operator $\mathcal{D}^{\star}_a$ satisfying the same
assumptions as $\mathcal{D}$ but now on $a \mathcal{X}$ then 
\begin{equation}
  \label{eq:52}
  \tau_{aX}(aX) = a^2 \tau_X(X)
\end{equation}
is a Stein kernel for $aX$. 
\end{proposition}
\begin{proof}
  The claim follows immediately from the definition. 
\end{proof}
Let $X_i, i=1, \ldots, n$, be independent random variables with
respective means $\nu_i$, and put $W= \sum_{i=1}^n X_i$.  Following
\cite[Lecture VI]{Stein1986} and
\cite{nourdin2013integration,nourdin2013entropy} we obtain an almost
sure  representation formula for the Stein kernel of sums of
independent random variables.

\begin{lemma} \label{couplinglemma} Suppose that (i)  $Id - \nu_i \in Im(\mathcal{T}_i)$ for
  $i=1, \ldots, n$ and (ii) $Id - \sum_{i=1}^n\nu_i \in Im(\mathcal{T}_W)$
and  (iii) the collection of functions of the form  $\mathcal{D}^{\star}g$
with   $g\in dom(\mathcal{D}, W, \tau_W) \cap \left(
    \bigcap_{i=1}^ndom(\mathcal{D}, X_i, \tau_{X_i}) \right)$ is dense
  in $L^1(\mu)$. 
Then 
  \begin{equation*}
    \tau_W(W) =  \E \left[ \sum_{i=1}^n \tau_{X_i} (X_i) \, | \, W
    \right] \quad a.s.
  \end{equation*}
\end{lemma}

\begin{proof}
For every $g \in dom(\cD,W,\tau_W)$ we have with \eqref{eq:8} that 
\begin{align*}
  - \E [ \tau_W (W)\cD^{\star} g(W) ] &= \E \left[ \left(W -\sum_{i=1}^n\nu_i\right) g(W) \right] \\
  &= \sum_{i=1}^n \E \left\{ \E [(X_i - \nu_i) g(W) |W_i ] \right\}
\end{align*}
where $W_i=W-X_i=\sum_{j\neq i}X_j$ is independent of
$X_i$. Therefore, conditionally on $W_i$ we can use (an appropriate
version of) \eqref{eq:8} for each $X_i$, turning the previous expression into
\begin{align*}
  - \sum_{i=1}^n \E \left\{ \E \left[\tau_{X_i}(X_i) \cD^{\star} g(W)
      |W_i \right] \right\}
  &=  -  \sum_{i=1}^n \E  \left\{ \E \left[\tau_{X_i}(X_i)  \cD^{\star} g(W) |W \right] \right\} \\
  &= - \E \left\{ \E \left[ \sum_{i=1}^n \tau_{X_i} (X_i) | W
    \right]\cD^{\star} g(W)\right\}
\end{align*}
where the first equality follows de-conditioning \mbox{w.r.t.} $W_i$
and then conditioning \mbox{w.r.t.} $W$.  The assertion follows by
denseness.
\end{proof} 

Combining this representation lemma with Corollary~\ref{cor:compstfac}
leads to the following general result, which in particular implies
inequality \eqref{eq:CLTSTk} from Section~\ref{sec:application-2-}.  

\begin{proposition} \label{genprop2}   Suppose that the assumptions in Lemma
  \ref{couplinglemma} are satisfied.  Let $X$ be a random variable
  with finite mean $\nu=\sum_{i=1}^n\nu_i$. If
  $g_h=\mathcal{T}_X^{-1}(h - \E [h(X)])/\tau_X \in dom(\cD,W,\tau_W) \cap
  dom(\cD,X,\tau_X)$ then
 \begin{align*}
   | \E h(X) - \E h(W) | &\le  || \cD g_h ||_\infty   \E
   \left|   \tau_X(W)  -    \sum_{i=1}^n \tau_{X_i} (X_i)   \right|  
 \end{align*}
 for all $h \in
\mathcal{H}$ a class of functions as in Corollary~\ref{cor:compstfac}.
\end{proposition}

\begin{proof}
Lemma \ref{couplinglemma} with Corollary~\ref{cor:compstfac} (whose conditions are satisfied) gives that 
\begin{align*}
&| \E h(X) - \E h(W) | \\
&\hspace{1cm}\leq || \cD g_h||_\infty \left|  \E [\tau_X(W) - \tau_W (W)] \right|\\
&\hspace{1cm}\le || \cD g_h||_\infty \E \left| \tau_X(W)- \E \left[ \sum_{i=1}^n  \tau_{X_i} (X_i) | W \right] \right|.
\end{align*} 
The assertion now follows by Jensen's inequality for conditional expectations. 
\end{proof}
\begin{proposition}\label{prop:sumwithscale}
  Let $W = \frac{1}{\sqrt n}\sum_{i=1}^n \xi_i$ with $\xi_i, i=1,
  \ldots, n$ centered  independent random variables 
  with $\mathcal{D}$-Stein kernels $\tau_i, i=1, \ldots, n$. Then   
\begin{equation}
  \label{eq:48}
  \tau_W(W) = \frac{1}{n} \sum_{i=1}^n \E \left[ \tau_i(\xi_i) \, | \,
  W\right]
\end{equation}
is a Stein kernel for $W$.  Furthermore the Stein discrepancy of $W$
satisfies
  \begin{equation}
    \label{eq:44}
    S(W) := \sqrt{\E \left[ \left( 1 - \tau_W(W) \right)^2 \right]}\le
    \frac{1}{n} \sqrt{ \sum_{i=1}^n\mbox{Var}(\tau_i(\xi_i))}. 
  \end{equation}

\end{proposition}
\begin{proof}
  Identity \eqref{eq:48} follows from a straightforward conditioning
  argument. To see \eqref{eq:44} note how under the assumptions of the
  proposition we have 
  \begin{align*}
    \E \left[ \left( 1 - \tau_W(W) \right)^2 \right] & = \E \left[
                                                       \left( \E \left[ 
                                                       \frac{1}{n}
                                                       \sum_{i=1}^n
                                                       (1-
                                                       \tau_i(\xi_i))
                                                       \, | \, W
                                                       \right]\right)^2
                                                       \right] \\
    & \le  \E \left[
                                                       \left(  
                                                       \frac{1}{n}
                                                       \sum_{i=1}^n
                                                       (1-
                                                       \tau_i(\xi_i))
                                                        \right)^2
                                                       \right] \\
&  \le \frac{1}{n^2} \sum_{i=1}^n\mbox{Var}(\tau_i(\xi_i)).
  \end{align*}
\end{proof}

Our general setup also caters for comparison of distributions with
Stein pair based on different linear operators $\mathcal{D}$; this has
already been explored in \cite{goldstein2013stein} for Beta
approximation of the Polya-Eggenberger distribution. Here we illustrate the
technique for Gaussian comparison in terms of Stein discrepancies.
\begin{proposition}\label{sec:sums-indep-rand}
  Let $\mathcal{D}$ be a linear operator satisfying the Assumptions
  from Section \ref{sec:setup}; let $l$ be as in Assumption
  \ref{ass:prodrule}.  Let $W$ be centered with variance $\sigma^2$, and
  $\mathcal{D}$-Stein pair $(\mathcal{T}_W, \mathcal{F}(W))$; let
  $\tau_W$ be the corresponding Stein kernel. Let $Z \sim
  \mathcal{N}(0, 1)$ and $ S(W)$ as above
 the Stein discrepancy between $W$ and $Z$. Then for all
  $g \in dom((\cdot)', Z) \cap dom(\mathcal{D}, W, \tau_W)$ we have
  \begin{equation}
    \label{eq:30}
        \left| \E \left[ g'(W) - W g(W) \right]  \right| \le  S(W) \|
        g'\| + \sigma^2\| g'(\cdot) - \mathcal{D}^{\star}g( \cdot - l)
        \|_{\infty} + \| g(\cdot-l) - g( \cdot)\|_{\infty}.
  \end{equation}
\end{proposition}
\begin{proof}
Applying Proposition \ref{prop:stek} to $\nu = 0$ we get 
\begin{equation}
  \label{eq:43}
   \E \left[ W g(W - l) \right] = \E \left[ \tau_W(W) \mathcal{D}^{\star}g(W-l) \right]
\end{equation}
for all $g \in dom(\mathcal{D}, W, \tau_W)$. If  furthermore
 ${g \in }dom((\cdot)', Z)$ then 
\begin{align*}
  \E \left[ g'(W) - W g(W)  \right] & = \E \left[ g'(W) - W g(W-l)
                                     \right] + \E \left[ W \left(
                                     g(W-l) - g(W) \right) \right] \\
  & = \E \left[ g'(W) - \tau_W(W) \mathcal{D}^{\star}g(W-l) \right] + \E \left[ W \left(
                                     g(W-l) - g(W) \right) \right] \\
& =  \E \left[ g'(W)(1 - \tau_W(W)) \right]  + \E \left[ \tau_W(W) \left( g'(W) -
  \mathcal{D}^{\star}g(W-l) \right) \right]\\
& \quad     + \E \left[ W \left(   g(W-l) - g(W) \right) \right].
\end{align*}
Applying Cauchy-Schwarz to the first summand in the last equality
yields the first summand of \eqref{eq:30}. To get the second summand
of \eqref{eq:30} note that $\tau_W(W) \ge 0$ almost surely (recall
Remark \ref{rem:positstek} so that
\begin{align*}
  \left|  \E \left[ \tau_W(W) \left( g'(W) -
  \mathcal{D}^{\star}g(W-l) \right) \right] \right| \le   
  \E \left[ \tau_W(W) \right]  \| \left( g'(\cdot) -
  \mathcal{D}^{\star}g(\cdot-l) \right) \|_{\infty}
\end{align*}
and now we use $\E \left[ \tau_W(W) \right]   = \mbox{Var}(W) = \sigma^2$.
The last term in \eqref{eq:30} follows by a similar reasoning. 
\end{proof}
As an illustration we now provide a Gaussian approximation bound in
Wasserstein distance under a Stein kernel assumption.
\begin{proposition}\label{sec:sums-indep-rand-1} Let $W$  be centered
  with variance $\sigma^2$ and support 
  in $\delta \Z$ for some $\delta>0$. Consider
  $\mathcal{D} = \delta^{-1}\Delta^+_{\delta}$ as in Example
  \ref{ex:discmgrid}.  Suppose that the assumptions in Lemma
  \ref{couplinglemma} are satisfied. Then
  \begin{equation}
    \label{eq:34}
d_{\mathrm{Wass}} (W, Z) \le  S(W)
    +(1+\sigma^2)\delta
  \end{equation}
with $d_{\mathrm{Wass}}(W, Z)$ the Wasserstein distance between the
laws of $W$ and $Z$. 
\end{proposition}
\begin{proof}
We aim to apply \eqref{eq:30}, with $g= g_h$ the classical solution to
the Gaussian Stein equation 
\begin{align*}
  g'(x) - xg(x) =  h(x) - \E h(Z)
\end{align*}
where $h$ is a Lipschitz function with constant 1. 
The properties of such $g$ are well understood,
see e.g.\ \cite[Lemma 2.3]{BC05}. In particular these functions are
differentiable and bounded   with  $\|g'\|_{\infty}\le 1$ so that 
\begin{equation*}
  |g(x-\delta) - g(x)| = \int_{-\delta}^0 g'(x+u) du  \le \delta
\end{equation*}
for all $x \in \R$. Also, $\|g''\|_{\infty}\le 2$  and  hence
\begin{align*}
   |g'(x) - \mathcal{D}^{\star}g( x - l)| & =  | g'(x) - \delta^{-1} (g(x)
                                            - g(x-\delta)) | \\
  & = \left| \frac{1}{\delta} \int_{-\delta}^0 \int_0^ug''(x+v)dv du
    \right| \\
& \le  \delta,
\end{align*}
again for all $x \in \R$. The claim follows. 
\end{proof}
Finally, following up on the results presented in
Section~\ref{sec:application-2-}, we conclude with a central limit
theorem for sums of centered Rademacher random variables. 
\begin{corollary}
  Let $W = \frac{1}{\sqrt n}\sum_{i=1}^n \xi_i$ with
  $\xi_i, i=1, \ldots, n$ independent centered with support in
  $\left\{ -1, 1 \right\}$. Fix $\mathcal{D}f = f(x+1) - f(x-1)$ {and} 
  let $\tau_i(\xi_i) = \mathbb{I}(\xi_i = 1)$. The
  $\tau_i(\xi_i)_{i=1, \ldots, n}$ are  $\mathcal{D}$-Stein
  kernels for $(\xi_i)_{i=1, \ldots, n}$ and 
  \begin{equation}
    \label{eq:47}
   d_{\mathrm{Wass}} (W, Z) \le  \frac{3}{\sqrt{n}}.
  \end{equation}
\end{corollary}
\begin{proof}
 The first claim is immediate. Next we use \eqref{eq:44} to deduce that 
\begin{align*}
 S(W)&  \le  \frac{1}{n}
       \sqrt{\sum_{i=1}^n\mbox{Var}(\tau_i(\xi_i))} =
       \frac{1/2}{\sqrt{n}}. 
\end{align*}
Finally  we apply   \eqref{eq:34} with $\sigma^2=1$ and $\delta =
  \frac{1}{\sqrt{n}}$. 
\end{proof}
\begin{remark}
  It is straightforward to extend the results of this Section to
  random sums of independent random variables and therefore deduce
  central limit theorems for randomly centered random variables. A
  much more challenging task is to deal with non-randomly centered
  random sums, as e.g.\ in \cite{dobler2012rates}.  
\end{remark}
\section{Stein bounds}
\label{sec:examples-again}

As anticipated, in this section we discuss several non-asymptotic
approximation via Stein differentiation in several concrete examples.
The main purpose of this Section is illustrative and most of the
examples we discuss lead to well-known situations. Relevant references
are given in the text.

\subsection{Binomial approximation to the Poisson-binomial distribution} 
\label{sec:binom-appr-poiss}

An immediate application of Proposition \ref{genprop2} can be found in
binomial approximation for a sum of independent Bernoulli random
variables. Writing $X$ for a ${\rm Bin}(n,p)$ and $W=\sum_{i=1}^nX_i$
with $X_i\sim {\rm Bin}(1,p_i)$, $i=1,\ldots,n$, and
$np=\sum_{i=1}^np_i$ (the distribution of $W$ is called a
Poisson-binomial distribution, see e.g. \cite{ehm1991binomial}), we
readily compute $$\tau_X(x)=(1-p)x \mbox{ and }
\tau_{X_i}(x)=(1-p_i)x.$$ Here we use $\cD = \Delta^+$, the forward
difference.  Thus for any measurable function $h$ such that $\E |
h(X) | < \infty$ and $\E |h(W)| <\infty$ ,
\begin{align}
| \E h(X) - \E h(W) |&\leq || \cD g_h||_\infty \E \left| (1-p)W-
  \sum_{i=1}^n  (1-p_i) X_i  \right| \nonumber \\
&\leq  ||  \cD g_h||_\infty \sum_{i=1}^n 
|p_i  -p | p_i  .  \label{expression1} 
\end{align}

 An alternative angle on this problem is to use the score
function approach, although here with $\mathcal{T}(Id)$ instead of
$\mathcal{T}(1)$.  It is easy to show (see e.g. Example
\ref{ex:listoftaus}.2.(b)) that
$$\mathcal{T}_{Bin(n,p)}(f)(x) =\frac{p(n-x)}{(1-p)(x+1)} f(x+1) - f(x)$$
so that for $f=Id$, the identity function, 
$$ \mathcal{T}_{Bin(n,p)}(Id)(x) =
 \frac{np-x}{1-p}.
$$ 
By Example \ref{ex:discmCONT} we find that $f=Id \in {\mathcal{F}}(X)
\cap {\mathcal{F}}(W)$ because $Id(0)=0$.  Now let $h $  be such that
$\E |h(X)|<\infty$ and $\E |h(W)|<\infty$, and let  $g_h
=\mathcal{T}_X^{-1}(h-\E h[X])/ {Id} $; then $g \in dom(\Delta^+,W, Id)
\cap dom(\Delta^+,X,Id)$. From \eqref{eq:start} we obtain that
$$
\E h (W)  - \E h (X)   = \E \left[
   g_h(W+ 1) \left\{ \mathcal{T}_{Bin(n,p)}(Id) (W) -
     \mathcal{T}_{\mathcal{L}(W)}  (Id)(W)  \right\}  \right] .
$$
By \eqref{eq:basicstein}, using the notation $g_a(x) = g(x+a)$ for a
function in $x$,  
\beas
\lefteqn{\E  g_h(W+ 1)  \mathcal{T}_{\mathcal{L}(W)} (Id) (W)}\\
  &=& - \E W\Delta^{-} g(W+ 1) \\
&=& - \sum_{i=1}^n \E \left[ \E \left\{ X_i  \Delta^{-} g_{\sum_{j \ne i} X_j+ 1} (X_i) | X_j, j \ne i \right\} \right] \\
&=& \sum_{i=1}^n \E \left[ \E \left\{  \mathcal{T}_{Bin(1,p_i)}(Id) (X_i)  g_{\sum_{j \ne i} X_j+ 1} (X_i) | X_j, j \ne i \right\} \right] \\
&=&  \sum_{i=1}^n \E \left\{ g(W+ 1)  \mathcal{T}_{Bin(1,p_i)} (Id)(X_i)  \right\}.
\enas 
Hence
\beas
\lefteqn{\E h (W)  - \E h (X)  }\\
& =&  \E \left[
   g_h(W+ 1)  \left\{ \mathcal{T}_{Bin(n,p)}(Id)(W) -   \sum_{i=1}^n
     \mathcal{T}_{Bin(1,p_i)}(Id) (X_i)   \right\}  \right]\\ 
&=& \E \left[
   g_h(W+ 1)   \left\{ \frac{np-W}{1-p} - \sum_{i=1}^n  \frac{p_i-X_i}{1-p_i} \right\}  \right]\\
&=& \E \left[   g_h(W+ 1) \sum_{i=1}^n (p_i - X_i)  \left\{ \frac{1}{1-p} -  \frac{1}{1-p_i} \right\}  \right]
\enas 
and so 
\begin{equation} \label{expression2}
  | \E h(X) - \E h(W) | \le   \frac{\| g_h\|_{\infty}}{1-p} \sum_{i=1}^n | p - p_i| 
\E  \left|  \frac{p_i-X_i}{1-p_i}\right| = \frac{ 2 \| g_h\|_{\infty}}{1-p} \sum_{i=1}^n | p - p_i| p_i .
\end{equation}
 The fact that we obtain two different bounds, {\eqref{expression1} and \eqref{expression2},} for the same problem illustrates the freedom of choice in specifying $f$ and $g$ in the Stein equation. 
In \cite{ehm1991binomial}, bounds for  $\sup_x| \cD \frac{g_h(x)}{x+1}
|$ are calculated, and in \cite{eichelsbacher2008stein} a bound for
$\sup_x|  \frac{g_h(x)}{x+1} |$ is given.  

\subsection{Distance between Gaussians}
\label{sec:gauss}

Consider two centered Gaussian random variables $X_1$ and $X_2$ with
respective variances $\sigma_1^2\le\sigma_2^2$, say. Denote $\phi$ the
density of $Z$, a standard normal random variable.  The canonical Stein
operators are then of the form
\begin{align*}
  \mathcal{T}_i f(x) = f'(x) - \frac{x}{\sigma_i^2} f(x)
\end{align*}
acting on the classes $\mathcal{F}_1(X_1) = \mathcal{F}_2(X_2) =
\mathcal{F}(Z)$ of $Z$-integrable differentiable functions such that
$\left( f \phi \right)' \in L^1(dx)$.  In this simple toy-setting it
is possible to write out \eqref{eq:start}   in full
generality. Indeed we have
\begin{align*}
 f_1 g_h & =   \mathcal{T}_1^{-1}(h-\E_1h) \\
 & =  
  e^{x^2/(2\sigma_1^2)} \int_{-\infty}^x (h(y) - \E h(X_1))
  e^{-y^2/(2\sigma_1^2)} dy \\
& = e^{\left( {x}/{\sigma_1} \right)^2/2} \sigma_1
\int_{-\infty}^{x/\sigma_1} \left( h(\sigma_1u) - \E h(\sigma_1Z
\right) e^{-u^2/2}du
  & =:\sigma_1g_{\tilde h,0}(x/\sigma_1)
\end{align*}
with $\tilde{h}(u) = h(\sigma_1 u)$ and  $g_{h,0}$ the solution of the classical Stein
equation given by 
\begin{equation*}
  g_{h,0}(x) = e^{x^2/2} \int_{-\infty}^x (h(y) - \E h(Z))
  e^{-y^2/2}dy. 
\end{equation*}
In the particular case where
one is interested in the total variation distance, then one only
considers $h : \R \to [0,1]$ Borel functions for {which} 
$
 \| g_{h,0}\| \le \sqrt{\frac{\pi}{2}}$ and $ \|  g_{h,0}' \|
 \le 2 
$  (see
e.g. \cite[Theorem 3.3.1]{NP11}).
In the rest of this section we focus on such $h$, although similar
results are available for $h = \mathbb{I}_{(-\infty, z]}$ (leading to
bounds on the Kolmogorov distance, see \cite[Lemma 2.3]{ChGoSh11}) and for  $h \in Lip(1)$ (leading to
bounds on the Wasserstein
distance, see \cite[Proposition 3.5.1]{NP11}). 
Identity \eqref{eq:start} becomes
\begin{align*}
  \E h(X_2) - \E h(X_1) & = \E \left[ (f_1(X_2) - f_2(X_2)) \left(
      \frac{\sigma_1 g_{\tilde{h},0}(X_2/\sigma_1)}{f_1(X_2)} \right)' \right. \\
&\quad \quad \left. + \left( \mathcal{T}_1f_1(X_2) - \mathcal{T}_2f_2(X_2) \right) \left(
      \frac{\sigma_1g_{\tilde{h},0}(X_2/\sigma_1)}{f_1(X_2)} \right) \right].
\end{align*}
 for any  $f_1, f_2 \in \mathcal{F}(Z)$.  There are many directions
 that can be taken from here, of which we illustrate three (to simplify notation we write
 $g_h$ for $g_{\tilde{h},0}$). 
 \begin{itemize}
\item Taking $f_1 = 1$ and $f_2 = 0$ (see Remark \ref{rem:takingf0})  leads to the identity 
  \begin{equation*}
     \E h(X_2) - \E h(X_1) = \E \left[ g_h' \left(
         \frac{X_2}{\sigma_1} \right) - \frac{X_2}{\sigma_1} g_h\left(
         \frac{X_2}{\sigma_1} \right)  \right] 
  \end{equation*}
because $\mathcal{T}_1(1)(x)  = -x/\sigma_1^2$. 
Recalling that $\E \left[ X_2 \zeta(X_2) \right]  =  \sigma_2^2 \E \left[
  \zeta'(X_2) \right]$ for any differentiable function $\zeta$, and also
noting  that one can interchange the roles of $X_1$ and $X_2$, we
deduce the bound 
\begin{equation} \label{gauss1}
  d_{\rm TV}(X_1, X_2) \le \frac{2}{\sigma_2^2} \left|
    \sigma_1^2-\sigma_2^2 \right|, 
\end{equation}
already obtained e.g. in \cite[Proposition 3.6.1]{NP11}. 
\item  Taking $f_1 = \sigma_1^2$ and $f_2 = \sigma_2^2$ (thus a
  particular case of the
  comparison of kernels from Corollary \ref{cor:compstfac}) {also} yields {\eqref{gauss1}}.
 \item Taking $f_1 = f_2 = 1$ (thus a particular case of the
   comparison of scores from Corollary~\ref{cor:compscor}) yields the
   identity
 \begin{equation*}
 \E h(X_2) - \E h(X_1)  = \E \left[ X_2 \left(
     \frac{1}{\sigma_1^2}-\frac{1}{\sigma_2^2} \right)  \left(
     {\sigma_1}g_{h,0}\left( \frac{X_2}{\sigma_1} \right)
   \right)\right] 
 \end{equation*}
because $\mathcal{T}_i(1)(x) = -x/\sigma_i^2$. Using $\E \left| X_2 \right|
= \sqrt{\frac{2}{\pi}}\sigma_2$ and
$\|\sigma_1 g_{h,0}(\cdot
/\sigma_1)\|_{\infty} \le \sigma_1\sqrt{\frac{\pi}{2}} $ leads to 
\begin{equation*}
  d_{\rm TV}(X_1, X_2) \le  \frac{\left|
     \sigma_1^2-\sigma_2^2 \right|}{\sigma_1 \sigma_2},
\end{equation*}
which is better than {\eqref{gauss1}} 
whenever $\sigma_2/\sigma_1
< 2$. 

\end{itemize}

\subsection{From Student to Gauss}
\label{sec:studgauss}


Set $X_1=Z$ standard Gaussian and $X_2=W_\nu$ a Student $t$ random
variable with $\nu>2$ degrees of freedom. In this case  the Stein
kernels for both distributions are well defined and given,
respectively, by  $\tau_1=1$
and $\tau_2(x)=\frac{x^2+\nu}{\nu-1}$, see Example \ref{ex:listoftaus}. 
All assumptions in Corollary~\ref{cor:compstfac} are satisfied so that
we can plug these functions with
$\mathcal{H}$ the class of Borel functions in $[0,1]$ to get 
\begin{equation}
  \label{eq:14}
  d_{\mathrm{TV}}(Z,W_{\nu}) \le 2 \E \left|\frac{W_{\nu}^2+\nu}{\nu -1} -
       1  \right|
\end{equation}
where, as in the previous example, we make use of our knowledge on the
solutions of the Gaussian Stein equation.  It is
straightforward to compute \eqref{eq:14} explicitly (under the
assumption $\nu>2$, otherwise the expectation does not exist) to get
\begin{equation}
  \label{eq:16}
  d_{\mathrm{TV}}(Z,W_{\nu}) \le  \frac{4}{\nu-2}.
\end{equation}
A similar result is obtained with 
Corollary~\ref{cor:compscor}, namely 
\begin{equation*}
  d_{\mathrm{TV}}(Z,W_{\nu}) \le \sqrt{\frac{\pi}{2}} 
   \frac{-2+8 \left( \frac{\nu}{1+\nu} \right)^{(1+ \nu)/2}}{(\nu-1)
     \sqrt \nu B(\nu/2, 1/2)}, 
\end{equation*}
which is of the same order as \eqref{eq:16}, with a better
constant, but arguably much less elegant. 
\begin{remark}
It is of course possible to exchange the roles of the Student and the
Gaussian in the above computations. 
\end{remark}
 

\subsection{Exponential approximation}
\label{sec:expon}

Let $X_{(n)}$ be the maximum of $n$ i.i.d.~uniform random variables on
$[0,1]$. It is known that $M_n=n(1-X_{(n)})$ converges in distribution
to $X_1$ a rate-1 exponential random variable.  Note that $\E [M_n] =
\frac{n}{n+1}\neq1$. In order to apply Corollary~\ref{cor:compstfac}
most easily we are led to consider the slightly transformed random
variable $ X_2 = \frac{n+1}{n}M_n = (n+1)(1-X_{(n)})$. 

The canonical operator for $X_1$ is  $\mathcal{T}_1 f = f'-f$ acting on the class of
differentiable $f$ such that $f(0) = 0$. The Stein equation
\eqref{eq:genstequationfg} becomes
  \begin{equation*}
    h(x) - \E h(X_1) = f(x) g'(x) + f'(x) g(x) - g(x) f(x) = (fg)'(x) - (fg)(x). 
\end{equation*}
Then the solution pairs $(f,g)= (f_h, g_h)$ are such that $(fg)(x) =\mathcal{T}_{\rm exp}^{-1}(h) $ so that
\begin{equation} \label{expsol} (f g) (x) = e^x \int_0^x (h(u) - \E
  h(X_1)) e^{-u}du\end{equation} for $x>0$.  If $h(x) = \mathbb{I}(x
\le t)$ we need to understand the properties of  
\begin{equation*}
  (fg) (x)= e^{-(t-x)^+} - e^{-t}.
\end{equation*}
This function is bounded
and differentiable on $\R$, with limit 0 at the left boundary and
constant with value 
$1-e^{-t}$ for all $x \ge t$ (see also \cite[Lemma 3.2]{ChFuRo11}).
Taking $g(x) = x^\epsilon$ in \eqref{expsol} the corresponding function $f$
from \eqref{expsol} is $$f_{t,
  \epsilon} (x) = x^{- \epsilon} \left( e^{-(t-x)^+} - e^{-t}
\right)$$ 
with $a^+ = \max (a, 0)$. 
For all choices  $0< \epsilon < 1$ we have 
\begin{equation*}
  \lim_{x \to 0} f_{t, \epsilon}(x) = 0 \mbox{ and } \lim_{x\to
    \infty}f_{t, \epsilon}(x) = 0 \mbox{ and }     \left\| f_{t,
      \epsilon}    \right\|_{\infty} = t^{-\epsilon}(1-e^{-t}), 
\end{equation*}
as well as 
$  \lim_{x\to 0}f_{t, 1}(x) = e^{-t}$ (see \cite{ChFuRo11} for details
on the cases $\epsilon=0$ and $\epsilon=1$).

We now turn our attention to the problem of approximating the law of
$X_2$, whose  density  is $  p(x)  = \frac{n}{n+1} ( 1- \frac{x}{n+1}
)^{n-1}$ with support $[0, n+1]$. Taking derivatives we get
$$\mathcal{T}_2 f (x) = f'(x) -\frac{n-1}{n+1-x} f(x)
$$
acting, as above, on the class of differentiable functions such that
$f(0) = 0$. Clearly $f_{t,  \epsilon}(0)g_\epsilon(0) = 0$ for all
$0<\epsilon<1$  and therefore
\begin{eqnarray*}
P(X_2\le t) - P(X_1\le t) &=& \E [ (f_{t, \epsilon}g_{\epsilon})'(X_2) - (f_{t, \epsilon}g_{\epsilon})(X_2)]\\
&=& \E \left[   (f_{t, \epsilon}g_{\epsilon})(X_2) \left\{ \frac{n-1}{n+1-X_2}  - 1 \right\}  \right].
\end{eqnarray*} 
which yields the non-uniform bound
\begin{equation}\label{eq:19}
  | P( X_1 \le t) - P (X_2 \le t)| 
\le   t^{-\epsilon}(1-e^{-t}) \E  \left[
  X_2^{\epsilon}  \left|\frac{n-1}{n+1-X_2}  - 1 \right| \right].
\end{equation} 
The quantity on the rhs of \eqref{eq:19} can be optimised numerically
in $(\epsilon, t)$. For example for $n = 100$ and $t=1/2$, we can
compute the upper bound at $\epsilon=0$ to get 0.00497143 and
0.00852033 at $\epsilon=1$. The optimal choice of $\epsilon$ in this
case is $\epsilon \approx 0.138$ for which the bound is 0.00488718.
Obviously, in this simple situation, it is also easy to evaluate the
expressions $\Delta(t) = \sup_t \left| P(X_2\le t) - P(X_1\le t)
\right|$ numerically; 
{explorations} show that there is some
interesting optimization (depending on the magnitude of $t$) to be
performed in order to obtain good bounds.


\subsection{Gumbel approximation}
\label{sec:gumbel}

Let $X_{(n)}$ be the maximum of $n$ \mbox{i.i.d.} exponential random
variables. It is known that $M_n=X_{(n)}-\log n$ converges in
distribution to $X_1$ a Gumbel random variable with density $p(x) =
e^{-x}e^{-e^{-x}}$ on $\R$. The Stein kernel of the Gumbel does not take
on a tractable form, hence we shall here rather use
Corollary~\ref{cor:compstfac} with another choice of function
$\omega$.

A natural choice for $\omega$ is the score function, here $u_{\rm
  Gumbel}(x)=e^{-x}-1$, since in this case 
$\mathcal{T}_{\rm Gumbel}^{-1}(u_{\rm Gumbel})=1$.
As for the exponential example, we here also run into the difficulty
that $\E [e^{-M_n}-1] = \frac{n}{n+1}-1\neq0$, leading us to consider
the transformed random variable $ X_2 = M_n+\log\frac{n}{n+1}.$ Simple
calculations give $\mathcal{T}_2^{-1}(e^{-x}-1)=1- \frac{e^{-x}}{n+1}$
and we can use Corollary~\ref{cor:compstfac} to obtain
\begin{align*}
  |\E h(X_2) - \E h(X_1)|& \leq || g_h'||_\infty \E \left| 1 - \left(  1- \frac{e^{-X_2}}{n+1} \right) \right| \\
& =   || g_h'||_\infty \frac{1}{n+1}\E e^{-X_2}
\end{align*}
with $g_h(x)=\mathcal{T}_{\rm Gumbel}^{-1}(h)$. Since, furthermore, $\E  e^{-X_{2}}  = 1$ we deduce 
\begin{equation*}
   |\E h(X_2) - \E h(X_1)| \le \frac{1}{n+1} || g_h'||_\infty.
\end{equation*}
Again it is easy to express $g_h$ explicitly in most cases. For
example,  taking $h(x) =
\mathbb{I}(x \le t)$ we readily compute $g_h(x) = e^x \left(
  e^{-(e^{-t}-e^{-x})^+}-e^{-e^{-t}} \right)$ which can be shown to
satisfy $\| g_h\| \le e^t(1-e^{-e^{-t}}) \le 1$ and $\| g_h' \| \le
1$. This provides  the uniform bound
\begin{align*}
|  P(X_2 \le t) - P(X_1 \le t)| \le \frac{1}{n+1},
\end{align*}
which is of comparable order (though with a worse constant) with,
e.g., \cite{hallwell1979rate}.  


\section*{Acknowledgements}
{This work has been initiated 
{when}
  Christophe Ley and Yvik Swan were visiting Keble College,
  Oxford. Substantial progress was also made during a stay at the CIRM
  in Luminy. Christophe Ley thanks the Fonds National de la Recherche
  Scientifique, Communaut\'{e} fran\c{c}aise de Belgique, for support
  via a Mandat de Charg\'{e} de Recherche FNRS.  Gesine Reinert was
  supported in part by EPSRC grant EP/K032402/1.  Yvik Swan gratefully
  acknowledges support from the IAP Research Network P7/06 of the
  Belgian State (Belgian Science Policy).  We thank C. Bartholm\'e for
  discussions which led to the application given in
  Section~\ref{sec:rates-conv-frech}.  The authors would further like
  to thank Oliver Johnson, Larry Goldstein, Giovanni Peccati and
  Christian D\"obler for the many discussions about Stein's method
  which have helped shape part of this work. In particular, we thank
  Larry for his input on Section~\ref{sec:conn-with-bias}, Christian
  for the idea behind Section~\ref{sec:distr-satisfy-diff} and Oliver
  for the impetus behind the computations shown in
  Section~\ref{sec:binom-appr-poiss}.}
 


\end{document}